\documentclass{amsart}
\usepackage{comment}
\usepackage[utf8]{inputenc}
\usepackage{color}
\usepackage{amsmath,amssymb}

\usepackage{mathtools}
\usepackage{mathrsfs}
\mathtoolsset{showonlyrefs,showmanualtags}

\usepackage{tikz}

\usepackage{caption}
\usepackage{subcaption}

\usepackage{scalefnt}

\usetikzlibrary{matrix}
\usetikzlibrary{decorations.text}
\usetikzlibrary{intersections}

\usetikzlibrary{mindmap,backgrounds}

\xdefinecolor{alert}{rgb}{0.7,0,0}
\xdefinecolor{princ}{RGB}{100,160,170}
\xdefinecolor{lavanda}{rgb}{0.8,0.64,1}
\xdefinecolor{rose}{rgb}{0.9,0.6,0.9}
\xdefinecolor{polvere}{rgb}{0.7,0.6,0.7}
\xdefinecolor{grigio}{rgb}{0.8,0.8,0.8}
\xdefinecolor{cipria}{RGB}{240,210,240}
\xdefinecolor{lyell}{RGB}{238,201,0} 
\xdefinecolor{ggg}{rgb}{0.2,0.2,0.3}
\xdefinecolor{dor}{rgb}{0.9,0.8,0}
\xdefinecolor{gia}{rgb}{0.88,0.68,0}
\xdefinecolor{dred}{rgb}{0.8,0,0}
\xdefinecolor{purple}{rgb}{0.6,0.4,0.8}
\xdefinecolor{darkpink}{RGB}{228,42,234} 
\xdefinecolor{dblue}{rgb}{0,0,0.15}
\xdefinecolor{darkblue}{rgb}{0,0,0.2}
\xdefinecolor{dgreen}{rgb}{0.1,0.7,0} 
\xdefinecolor{turq}{rgb}{0,0.4,0.8} 
\xdefinecolor{lb}{rgb}{0,0.6,1} 
\xdefinecolor{straw}{RGB}{235,199,227}

\DeclareSymbolFont{eulerletters}{U}{zeur}{b}{n}
\DeclareMathSymbol{\eulw}{\mathord}{eulerletters}{`w}
\DeclareMathSymbol{\eulU}{\mathord}{eulerletters}{`U}
\DeclareMathSymbol{\eulO}{\mathord}{eulerletters}{`O}
\DeclareMathSymbol{\eulF}{\mathord}{eulerletters}{`F}

\newcommand{\qo}{\text{ a.e. }}
\newcommand{\uD}{{\rm D}}
\newcommand{\ud}{{\rm d}}
\newcommand{\abs}[1]{\vert {#1}\vert}
\newcommand{\wh}{\widehat}
\newcommand{\wt}{\widetilde}
\newcommand{\wb}{\boldsymbol}
\newcommand{\wq}{\wh q}
\newcommand{\wx}{\wh q}

\newcommand{\wxu}{\wh q_1}
\newcommand{\wpsi}[1]{\wh\psi_{#1}}
\newcommand{\R}{\mathbb{R}}

\newcommand{\wxi}{\widehat{\xi}}

\newcommand{\wu}{\widehat{u}}

\newcommand{\cV}{\mathcal{V}}
\newcommand{\cW}{\mathcal{W}}
\newcommand{\bcW}{\boldsymbol{\cW}}
\newcommand{\bcV}{\boldsymbol{\cV}}
\newcommand{\bJ}{\boldsymbol{J\se}}
\newcommand{\wla}{\widehat{\lambda}}
\newcommand{\wtau}{\widehat{\tau}}
\newcommand{\well}{\widehat{\ell}}

\newcommand{\cF}{{\mathcal F}} 
\newcommand{\cG}{{\mathcal G}} 
\newcommand{\wF}{H}
\newcommand{\vwF}{\vec{\wF}}

\newcommand{\cH}{{\mathcal H}} 
\newcommand{\cK}{{\mathcal K}} 
 
\newcommand{\vPhi}[1]{{\vec\Phi}^{#1}} %
\newcommand{\vF}[1]{{{\vec F}_{#1}}} %
\renewcommand{\AA}[2]{\big( L_{\dot{g^1_{#1}}} \big(\wh\gamma_T + \wS{T}^0 - \wS{#1}^0\big) + L_{g^0_{#1}}\wpsi{#1}\big)(#2) }
\newcommand{\RR}[1]{ 
	L_{ \liebr{\dot{g^1_{#1}}}{g^1_{#1}}   }  \big(\wh\gamma_T + \wS{T}^0 - \wS{#1}^0\big)(\wxu) + L_{\dot{g^1_{#1}}}\wpsi{#1}(\wxu) 
	-  L_{g^1_{#1}}L_{g^0_{#1}}\wpsi{#1}(\wxu)}

\newcommand{\Rtre}{\mathfrak{r}_3}
\newcommand{\attr}{\rho}
\newcommand{\Hex}{G\se}

\newcommand{\se} {^{\prime \prime}} %
\newcommand{\dl}{{\delta\ell}} %
\newcommand{\de}{{\delta e}} %
\newcommand{\dz}{\delta z} %

\newcommand{\wS}[1]{\widehat S_{#1}} %
\newcommand{\wSinv}[1]{\widehat S^{-1}_{#1}} %

\newcommand{\liebr}[2]{\left[  {#1}, {#2} \right] } %
\newcommand{\liede}[3]{ L_{#1}{#2} \big( {#3}\big) } %
\newcommand{\liededo}[3]{ L^2_{#1}{#2} \big( {#3}\big) } %

\newcommand{\liedede}[4]{ L_{#1}L_{#2}{#3}\big( {#4}\big)  } %
\newcommand{\lieder}[2]{L_{#1}{#2}(\wxu)} %
\newcommand{\liederder}[3]{ L_{#1}L_{#2}{#3}(\wxu)  } %
\newcommand{\liederdo}[2]{ L^2_{#1}{#2} ( \wxu ) } %

\newcommand{\scal}[2]{\langle {#1} \, , \, {#2} \rangle} %

\newcommand{\ep}{\varepsilon}

\newcommand{\cO}{{\mathcal O}} %
\newcommand{\cU}{{\mathcal U}} %

\renewcommand{\paragraph}[1]{\smallskip \noindent {\bf #1}}

\newtheorem{lemma}{Lemma}
\newtheorem{proposition}{Proposition}
\newtheorem{theorem}{Theorem}
\newtheorem{remark}{Remark}
\newtheorem{definition}{Definition}

\newtheorem{ass}{Assumption}

\numberwithin{equation}{section}
\numberwithin{theorem}{section}
\numberwithin{lemma}{section}
\numberwithin{proposition}{section}
\numberwithin{remark}{section}

\usepackage{mathtools}
\mathtoolsset{showonlyrefs,showmanualtags}

\usepackage{soul}

\begin{document}

\title{Singular extremals in $L^1$-optimal control problems: sufficient optimality conditions}
\thanks{This work has been supported by  "National Group for Mathematical Analysis, Probability and their Applications" (GNAMPA-INdAM), by UTLN–Appel à projet “Chercheurs	invit\'es”, Universit\'e de Toulon, by Progetto Internazionalizzazione, Universit\`a degli Studi di Firenze and by CARTT - IUT de Toulon.}
\author{Francesca C. Chittaro}\address{Universit\'e de Toulon, Aix Marseille Univ, CNRS, LIS, Marseille, France}
\author{Laura Poggiolini}\address{DiMaI, Universit\`a di Firenze, 50139 Firenze, Italy}
%
%
\begin{abstract}
In this paper we are concerned with generalised $L^1$-minimisation problems, i.e.~Bolza problems involving the absolute value of the control with a control-affine dynamics. We establish sufficient conditions for the strong local optimality of extremals given by the concatenation of bang, singular and inactive (zero) arcs. The sufficiency of such conditions is proved by means of Hamiltonian methods. As a by-product of the result, we provide an explicit invariant formula for the second variation along the singular arc. 
 \end{abstract}

\maketitle

\section{Introduction}

$L^1$-minimisation problems, that is, optimal control problems aiming to minimise the $L^1$ norm of the control, have shown to model  quite effectively fuel consumption optimisation problems in engineering \cite{ChenCC16,Chen16,Ross06,Sakawa99} and some systems appearing in neurobiology \cite{Berret08}. In a recent paper dealing with the problem of minimising the  fuel consumption of an academic vehicle \cite{boizot}, 
a generalised version of this problem is studied:  the cost to be minimised is  the {\it absolute work}, modelled as the integral of the absolute value of the control, weighted
by the absolute value of a function dependent on the state (such kind of problems have been called \emph{generalised $L^1$ optimal control problems}
in \cite{CP19}).

Besides the well known solutions given by concatenations of singular and bang arcs, generalised $L^1$-optimal control problems are known to present
a new category of extremal arcs, in which the control 
 is neither  singular nor bang, that is, its value is uniquely determined by Pontryagin Maximum Principle, but 
is not an extremum point of the control set.
In these arcs the control is identically zero, 
so they are commonly called {\it zero arcs, zero thrust arcs, inactivated arcs} or {\it cost arcs}.
The property of $L^1$ minimisation of generating zero arcs is 
well known and exploited in practical situations, see for instance 
\cite{NQN}.

The Pontryagin Maximum Principle, suitably generalised for non-smooth frameworks (see e.g. \cite{Cla83}), provides
a set of necessary conditions for the optimality of admissible trajectories. The aim of this paper is to give a set of sufficient conditions
for the optimality of admissible trajectories that satisfy the Pontryagin Maximum Principle and whose associated extremal contains bang, singular and
inactivated arcs.

Before proceeding with further discussion, let us state the problem in detail. Let $M$ be a smooth $n$-dimensional manifold, $\psi \colon M\to \mathbb{R}$ a smooth function, $f_0$, $f_1$
 two smooth vector fields on $M$.  Finally let $T>0$ be fixed and let $q_0$, $q_T$ be two given points in $M$.
Consider a Bolza optimal control problem of the following form:
\begin{equation} \label{eq: cost}
	\mbox{minimise} \int_0^T |u(t) \psi(\xi(t))| \;\ud t 
\end{equation}
over all absolutely continuous trajectories that are Carath\'eodory solutions of the boundary value problem  {\it (admissible trajectories)} 
\begin{equation} \label{eq: contr sys}
	\begin{cases}
		\dot{\xi}(t) = \left( f_0 +  u(t) f_1 \right) \circ \xi(t) ,  \\
\xi(0)= q_0, \quad \xi(T) = q_T, \\
u \in L^\infty([0,T], [-1,1]).
	\end{cases}
\end{equation}
For a solution $\wxi$ of the Cauchy problem~\eqref{eq: contr sys}, according to \cite{ASZ98b}, we adopt the following notion of optimality.
\begin{definition}[Strong local optimality]
	The trajectory $\wxi$ is a strong local minimiser of problem \eqref{eq: cost}-\eqref{eq: contr sys}  if there
	exists a neighbourhood $\eulO$ of its graph in $\R \times M$  such that $\wxi$ is a minimiser among
	the admissible trajectories whose graphs are contained in $\eulO$ , independently of the
	values of the associated control function. We say that $\wxi$ 
	is a strict   strong local minimiser if it is the only minimising trajectory whose graph is  in $\eulO$.
\end{definition}

Pontryagin Maximum Principle states that, if $\wxi$ is locally optimal, then it must be an {\em extremal trajectory}, 
i.e.~the projection on the manifold $M$ of  the solution of a suitable Hamiltonian system defined on the cotangent bundle $T^*M$,
see e.g.~\cite{AS04,Cla83}. If we limit ourselves to consider normal extremals only, then such Hamiltonian system
is determined by the  values of the two Hamiltonian functions $\Phi^{\pm}(\ell) = 
\scal{\ell}{f_1(\pi\ell)} \pm|\psi(\pi\ell)|,\ \ell \in T^*M$, whose role is
analogous to the one played by the switching functions in 
smooth control-affine  optimal control problems. 
In particular, when neither of the two functions vanishes, 
the maximised Hamiltonian is realised by one and only one admissible control value, 
which, according to the sign of these two functions, is either $\pm 1$ or $0$. 
In the first case, in analogy with the control-affine case, we say that we have a {\em regular bang control},  
while in the latter case we say that we have a {\em zero} or {\em  inactivated control}; 
on the other hand, if  either $\Phi^+(\ell)$ or $\Phi^-(\ell)$ vanishes along a nontrivial arc of a Pontryagin extremal, 
then only the sign of the control is prescribed, and we say that we have a {\em singular control}. 

In \cite{boizot} the authors consider a specific parameter-dependent problem fitting in the class of problem 
\eqref{eq: cost}--\eqref{eq: contr sys}. They show that, according to the values of the parameters, the optimal extremal trajectories are
given by the concatenations of bang-zero-bang arcs or 
of bang-singular-zero-bang arcs.  Inspired by this result, we look for sufficient optimality conditions for such extremals: in  \cite{CP17,CP19}, we 
focus on extremals made by
a concatenation of bang-zero-bang arcs; 
here we consider the case of a concatenation of  bang-singular-zero-bang arcs and provide an adequate set of sufficient conditions.

Our approach, successfully applied also for other classes of problems (see for instance \cite{ASZ98a,ASZ02b,PS04,PS11,PSp15,CS16,SZ16}) 
relies on Hamiltonian methods. 
The main steps in these methods are the following:
\begin{itemize}
	\item  if the flow generated by the maximised Hamiltonian is well defined, and the maximised Hamiltonian is at least $C^2$,   find a Lagrangian submanifold $\Lambda_1$ of the cotangent bundle that projects injectively onto the base manifold, and such that 
at each time $t \in [0,T]$ its image under the flow generated by the maximised Hamiltonian projects 
one-to-one onto a neighbourhood of $\wxi(t)$ in the base manifold; 
\item thanks to  the local invertibility of this flow, lift  all admissible trajectories, with graph belonging to the neighbourhood of the graph of the reference one, to $\Lambda_1$; 
\item estimate the cost associated with every trajectory,  by means of a line integral  in $\Lambda_1$. 
\end{itemize}

As long as this construction is possible  - that is, as long as the image of the manifold via the 
flow generated by the maximised Hamiltonian is invertible  - it is possible 
to show that the reference trajectory is indeed a local minimiser.

In the series of papers \cite{ASZ02b, PS04, Ste07, CS10, PS11}, the existence of a suitable manifold $\Lambda_1$ is shown to be related to the coerciveness of the second variation associated with some sub-problem of the original one. Indeed, the second variation is written as an accessory linear-quadratic control problem on the tangent space; by the classical theory of linear-quadratic systems (see e.g.  \cite[Theorem 2.6]{SZ97}), if the second variation is coercive, then the image, under the linearised flow, of the space of 
transversality conditions of the accessory problem projects injectively onto the base manifold. A good candidate for the manifold $\Lambda_1$ is 
thus given by the image of such 
manifold of transversality conditions, under a suitably defined symplectic (or anti-symplectic) isomorphism. 

We stress that one of the main strengths of this approach relies on the fact that all trajectories whose graph is close to the reference one can be lifted to 
$\Lambda_1$, regardless of the value of the associated control, thus yielding optimality in the strong topology.

In the case of bang-bang extremals, the maximised Hamiltonian is not $C^2$. Nevertheless, the Hamiltonian methods described above can be applied  with minor adjustments, see e.g.~\cite{ASZ02b, PS04, Pog06, PSp08, PSp11}. 
When the extremal is singular or is the concatenation of bang and singular arcs, then a more sophisticated construction is required. 
In particular, it is no longer possible to use the flow associated with the maximised Hamiltonian of the control system 
 (see Section~\ref{sec:  over max Ham} for details).  However, 
as observed for the first time 
in \cite{Ste07}, Hamiltonian methods work even if the maximised Hamiltonian is replaced by a suitable {\em over-maximised Hamiltonian}. In the present 
paper, this construction is made possible thanks to a set of \emph{regularity assumptions} (Assumptions~\ref{ass:  nonzero at switch}-\ref{ass:  sglc}), holding along the reference trajectory.

The problem under study presents another tough issue:  indeed, we are dealing with a Bolza problem containing a singular arc. As it happens, for instance,
for singular extremals of the minimum-time problem or of a Mayer problem, the second variation is degenerate, thus not coercive. This problem can be 
surmounted by means of a Goh transformation, which provides a non-degenerate second variation, defined on a larger Hilbert space. Moreover, 
since we are dealing with a Bolza problem,
the construction of such second variation is particularly elaborate and, 
up to the authors' knowledge, this is the first time it is computed, at least in the invariant form we are using.

The paper is organised as follows: in Section~\ref{sec:notation} we state the notations that we are going to use throughout the paper; in Section~\ref{sec:regularityass} we state the first part of our assumptions. In Section~\ref{sec:extsecvar} we define the extended second variation. In Section~\ref{sec:overmaxi}, by using the regularity assumptions, we construct the over-maximised flow and in Section~\ref{sec:invertibility} we prove that the projected over-maximised flow emanating from an appropriate Lagrangian manifold is locally invertible. Finally, in Section~\ref{sec:  main results} we state and prove the main result of the paper. The result is illustrated in Section~\ref{sec:example} with an example for which explicit analytical computations can be done.  For the sake of readability, some technical computations are postponed to the Appendices.

\section{Notations}\label{sec:notation}

We denote with $TM$  and with $T^*M$ the tangent bundle and the cotangent bundle to $M$, respectively.
$\pi$ denotes the canonical projection of $T^*M$ on $M$;  the elements of $T^*M$ are denoted with $\ell$. 

In the following, small letters $f,g,k$ denote vector fields on the manifold $M$, and the corresponding capital letters are used to denote
the corresponding Hamiltonian lift, i.e. $F(\ell)=\langle \ell,f(\pi \ell)\rangle$.
Given a vector field $f$ on $M$, the  Lie derivative at a point $q\in M$ of  a smooth function $\varphi \colon M \to \R$ with respect to $f$ is denoted with
$L_f \varphi(q)= \langle d\varphi(q),f(q)\rangle$, and $L_f^2\varphi(q)=L_f\big(L_f\varphi\big)(q)$.
The Lie bracket of two vector fields $f,g$ is denoted as commonly with $[f,g]$.
In particular, for Lie brackets of indexed vector fields as $f_0,f_1$, we adopt the following notations:
\begin{equation}
f_{ij}=[f_i,f_j],\qquad f_{ijk}=[f_i,f_{jk}]. 
\end{equation}
Analogously, $F_{ij}(\ell)=\langle \ell,f_{ij}(\pi \ell)\rangle$ and $F_{ijk}(\ell)=\langle \ell,f_{ijk}(\pi \ell)\rangle$.

The symbol $\varsigma$ denotes the Poincar\'e-Cartan invariant on $T^*M$,  defined as  $\varsigma_{\ell}=\ell\circ \pi_* \ \forall \ell\in T^*M$.
The symbol $\boldsymbol{\sigma}_\ell =d\varsigma_\ell$ denotes the canonical symplectic form on $T^*M$. 
With each Hamiltonian function $F$ we associate the Hamiltonian vector field $\vec{F}$ on $T^*M$
defined by 
\begin{equation}
\scal{d F(\ell)}{\cdot} = \boldsymbol{\sigma}_{\ell}(\cdot,\vec{F}(\ell)).
\end{equation}
In this paper, a special role is played by the switching time $\wtau_1$ between the first bang arc and the singular one. 
Thus, we shall always consider flows starting from time $\wtau_1$,  evolving backwards in time up to time $t=0$  or forward in time up to time $t=T$. 
Capital cursive letters are used for the Hamiltonian flows associated with some Hamiltonian vector fields: for instance, $\cF_t$ denotes
the  flow,  from time $\wtau_1$ to time $t$, associated with the Hamiltonian vector field $\vec{F}$.

\smallskip
Throughout the paper, the superscript $\widehat{\cdot}\, $ is used for objects related to the reference trajectory $\wh{\xi}$. In particular,
$\wu(t)$ denotes the control associated with $\wxi$, $\widehat{f}_t$ the vector field $f_0+\wu(t) f_1$ and 
$\wh S_t$ its flow, i.e. $\wh S_t(q)$ is the solution at time $t$ of the Cauchy problem
\begin{equation}
\begin{cases}
\dot{\xi}(t)=\wh{f}_t\circ \xi(t),\\
\xi(\wtau_1)=q.
\end{cases}
\end{equation}
Analogously, $\widehat{F}_t(\ell)=\scal{ \ell}{\wh{f}_t(\pi\ell)}$ is the Hamiltonian associated with $\wh{f}_t$, and 
$\wh{\cF}_t$  denotes its associated  Hamiltonian flow.
Finally, we define the function $\wpsi{t}\colon M\to \R$  as $\wpsi{t}=\psi\circ \wh{S}_t$. 

\section{Regularity assumptions}\label{sec:regularityass}
 We consider an admissible trajectory  $\wxi$  of the control system~\eqref{eq: contr sys}
whose associated control $\wu$ has the following structure:
\begin{equation} \label{eq: struttura controllo}
\wu(t)=
\begin{cases}
1 & t\in(0,\wtau_1), \\
\wu_S(t)\in (0,1) & t\in(\wtau_1,\wtau_2),\\
0 & t\in(\wtau_2,\wtau_3),\\
-1 & t\in(\wtau_3,T).
\end{cases}
\end{equation}
where  $0<\wtau_1<\wtau_2<\wtau_3<T$.
$\wxi$ is called the \emph{reference trajectory}. The times $\wtau_i,\ i=1,2,3$ are called
 \emph{(reference) switching times} and,
analogously, the points  $\wx_i=\wxi(\wtau_i)$ are called \emph{(reference) switching points}.

This section is devoted to the statement of the necessary conditions for optimality and the discussion of the regularity assumptions along the reference trajectory.

\begin{ass} \label{ass:  nonzero at switch}
	Along the reference trajectory,	the cost $\psi$ does not vanish for $t\in[\wtau_1,\wtau_2]$ and 
	$t=\wtau_3$. Without loss of generality, we assume that $\psi\circ\wxi|_{[\wtau_1,\wtau_2]}>0$.
\end{ass}

\begin{ass} \label{ass:  non vanish}
For every  $t\in[0,T]$ such that $\psi(\wxi(t))=0$, it holds $L_{\dot{\wxi}(t)}\psi(\wxi(t))\neq 0$.  
\end{ass}
Assumptions~\ref{ass:  nonzero at switch}-\ref{ass:  non vanish} deal with the behaviour of the
function $\psi$; in particular, a direct consequence of Assumption~\ref{ass:  non vanish} is that  
the zeroes of the cost function $\psi$ along $\wxi$ are isolated, thus finite.
Since the zeroes of $\psi$ are the only non-smoothness points of the  
running cost, we can apply a classical version of PMP. Indeed,  
due to  non-smoothness, 
more general versions of 
PMP (see, e.g.,  \cite[Theorem 22.2]{Cla83}, \cite{kipka}) would be required. Nevertheless, thanks to Assumption~\ref{ass:  non vanish}, we can 
 rearrange
the optimal control problem into
a hybrid control problem, as defined in \cite[Section~22.5]{Cla83}, setting the  surface
$S=\{(t, x, y)  \colon   \psi(x) = 0, y = x\}$ as switching surface.  The analogue of the classical PMP for hybrid optimal control problems is 
the Hybrid Maximum Principle (see \cite[Theorem 17.4.1]{Sussmann99} and \cite[Theorem~22.26]{Cla83}) which, under Assumption~\ref{ass:  non vanish},  reduces to the standard smooth version of PMP,  see\cite{AS04}.
For this reason, here below we refer to the classical notions of Pontryagin extremal and PMP.

\begin{definition}[Pontryagin extremals]
For every $u \in [-1,1]$ and for $p_0 \in \{0,1 \}$, we consider the pre-Hamiltonian function associated with the optimal control problem \eqref{eq: cost}-\eqref{eq: contr sys}
\begin{equation} \label{eq: preHam}
h(\ell, u)  := F_0(\ell) + u F_1(\ell) - p_0 \abs{u \, \psi(\pi\ell)},
\end{equation}	
and we define the maximised Hamiltonian as
\begin{equation} \label{eq: Hmax}
H_{\max}(\ell)= \max_{u\in[-1,1]} h(\ell,u).
\end{equation}
Let $\lambda \colon [0,T] \to T^*M$ be an  absolutely continuous curve  such that $\psi(\pi\lambda(t))$  vanishes only for a finite number of times $0< t_1 < \ldots < t_k <T $. 
$\lambda$ is called a Pontryagin extremal of problem \eqref{eq: cost}-\eqref{eq: contr sys},  if there exist a constant $p_0 \in \{ 0, 1\}$ and an admissible control $u(t)$ (called {\em extremal control}) such that
	\begin{subequations} \label{PMP}
	\begin{alignat}{2}
        & (\lambda(t),p_0) \neq 0, \qquad && \ \forall \ t \in [0,T], \label{PMP nonzero} \\
	& \dot\lambda(t) = \vec{ h}(\lambda(t), u(t)), \qquad && \qo t \in [0,T], \label{PMP equation} \\
	&  h(\lambda(t), u(t)) = H_{\max} (\lambda(t)), \qquad && \qo  t \in [0,T],  \label{PMP max} \\
	& \pi\lambda(0) = q_0, \quad  \pi\lambda(T) =q_T. \label{PMP endpoints}	
	\end{alignat}
	\end{subequations}

If $p_0 =1$, then $\lambda$ is called a {\em normal Pontryagin extremal}, if $p_0 = 0$ we say that $\lambda$ is an {\em abnormal Pontryagin extremal}.
	\end{definition}

As discussed in the Introduction, in the case of normal extremals, the Hamiltonian functions
\begin{equation}
\Phi^\pm(\ell)  := F_1(\ell) \pm \abs{\psi(\pi\ell)} 
\end{equation}
play the same role as the switching functions do in control-affine problems. Indeed, if both $\Phi^-$ and $\Phi^+$ are non-zero along an extremal, then  the extremal control is uniquely determined by equation~\eqref{PMP max}: in particular, it is zero if $\Phi^-$ and $\Phi^+$ have 
different signs, it is $+1$ if both are positive, and it is $-1$ if both are negative.
If only one between $\Phi^-$ and $\Phi^+$ is zero, then PMP prescribes only the sign of the extremal control. The last case, where both $\Phi^-$ and $\Phi^+$
are zero, is highly degenerate and in this case equation~\eqref{PMP max} gives no information about the value of the extremal control.
For these reasons, $\Phi^-$ and $\Phi^+$ are called \emph{switching functions}. Subarcs of a Pontryagin extremal are thus classified according to the signs of the switching functions.

\begin{definition} \label{def: reg}
Let $\lambda \colon [0,T]\to T^*M$ be a normal Pontryagin extremal for Problem~\eqref{eq: cost}-\eqref{eq: contr sys}, and let $I\subset[0,T]$ be an open interval. 

If $\Phi^-(\lambda(t))\Phi^+(\lambda(t))>0$ for every $t\in I$, then $\lambda|_{I}$ is called a \emph{regular bang arc}.  

If $\Phi^-(\lambda(t))\Phi^+(\lambda(t))<0$ for every $t\in I$, then $\lambda|_{I}$ is called an \emph{inactivated} or \emph{zero arc}. 

If $\Phi^-(\lambda(t))\Phi^+(\lambda(t))=0$ for every $t\in I$,
then $\lambda|_{I}$ is called a \emph{singular arc}. 

In particular, if  one between $\Phi^-(\lambda(t))$ and $\Phi^+(\lambda(t))$ is different from zero for every $t \in I$, then  $\lambda|_I$ is a \emph{ non-degenerate singular arc}. Else we call it \emph{degenerate}.
\end{definition}

\begin{ass} \label{ass:  struttura arco}
There exists a normal Pontryagin extremal $\wla$ associated with the reference control $\wu$ such that $\pi\wla(t) =\wxi(t)$ for every $t \in [0,T]$. 
We assume that $\wla|_{[0,\wtau_1)}$ and $\wla|_{(\wtau_3,T]}$ are regular bang arcs, that $\wla|_{(\wtau_1,\wtau_2)}$ 
is a non-degenerate~\footnote{  We  point out that degenerate singular arcs occur if and only if for some $t \in (\wtau_1, \wtau_2)$ $F_1(\lambda(t)) = \psi(\pi\lambda(t)) = 0$; for the extremal $\wla$ this situation  is precluded by Assumption \ref{ass:  nonzero at switch}.} singular arc, and that
$\wla|_{(\wtau_2,\wtau_3)}$ is an inactivated arc.
\end{ass}

\noindent
{\bf Notation. } {\it
We set $\well_i=\wla(\wtau_i)$, for $i=0,1,2,3$,  
and $\well_T=\wla(T)$.

Moreover, we define the following constants:
\begin{equation}
a_i=\mbox{sign}\big(\psi(\wx_i)\big),\ i=1,2,3. 
\end{equation}
Thanks to Assumption~\ref{ass:  nonzero at switch}, it follows that $a_1=a_2=1$.
}

\smallskip

\begin{remark}
	We recall  that, if the reference extremal is optimal, then it must satisfy PMP and the switching functions must satisfy the mild version of 
the inequalities appearing in Definition~\ref{def: reg}; the only additional requirements in Assumption~\ref{ass:  struttura arco} are the regularity of the arcs and
the fact that the extremal is normal.
\end{remark}%

\begin{remark} \label{rem: zeropsi}
Assumption~\ref{ass:  nonzero at switch} ensures that, in a neighbourhood of $\wla([0,T])$, the switching surfaces $\{\Phi^+=0\}$ and $\{\Phi^-=0\}$ 
do not intersect each other.
Together with the fact that the zeroes of $\psi$ along $\wxi$ are finite, this fact
guarantees that the Hamiltonian vector field associated with $\wh F_t$ 
is well defined along the reference trajectory, but for at most a finite number of times, that is, the switching times and 
the  zeroes of $\psi(\pi\wla(t))$. 
\end{remark}

Assumption~\ref{ass:  struttura arco} and equation \eqref{eq: struttura controllo} imply the following conditions on the sign of the switching functions along the reference extremal:
\begin{alignat}{2}
\Phi^-(\wla(t))&> 0  \qquad &&t\in [0,\wtau_1), \label{eq: bang 1 reg}\\
\Phi^-(\wla(t))&= 0  \qquad &&t\in [\wtau_1,\wtau_2],\label{eq: sing} \\
 \Phi^-(\wla(t))< 0&< \Phi^+(\wla(t)) \qquad &&t\in (\wtau_2,\wtau_3), \label{eq: zero} \\
\Phi^+(\wla(t))&<0 \qquad &&t\in(\wtau_3,T]. \label{eq: bang 4 reg}
\end{alignat}  
Equations~\eqref{eq: bang 1 reg}--\eqref{eq: bang 4 reg} yield a set of \emph{higher order necessary conditions}. Indeed, combining equations~\eqref{eq: bang 1 reg} 
and \eqref{eq: sing}, 
we obtain that 
$\frac{d}{dt}\Phi^-(\wla(t)) \vert_{t = \wtau_1^-} \leq 0$, while \eqref{eq: sing} gives
$\frac{d}{dt}\Phi^-(\wla(t))=\frac{d^2}{dt^2}\Phi^-(\wla(t))=0$  for every $t \in(\wtau_1,\wtau_2)$. Explicit computations show that $\frac{d}{dt}\Phi^-(\wla(t)) \vert_{t = \wtau_1^\pm}=\{F_0,\Phi^-\}(\well_1)$. 
By continuity 
this implies that $\left.\frac{d}{dt}\Phi^-(\wla(t)) \right\vert_{t = \wtau_1}=0$,
so that we must have $\frac{d^2}{dt^2}\Phi^-(\wla(t))\vert_{t = \wtau_1^-} \geq 0$.

Analogously, from equations \eqref{eq: sing} and \eqref{eq: zero}, we obtain that
$\frac{d^2}{dt^2}\Phi^-(\wla(t))|_{t=\wtau_2^+}\leq 0$.

At time $\wtau_3$, 
$\Phi^+(\wla(t))$ is differentiable and changes sign from positive to negative, thus $\frac{d}{dt}\Phi^+(\wla(t))|_{t=\wtau_3}\leq 0$.

The regularity assumptions at the switching points consist in a strengthening of the above inequalities.

\begin{ass}[Regularity at the switching points.] \label{ass:  reg switch}
\begin{align}
 \left( F_{001} + F_{101} \right)(\well_1) 
+ a_1\liede{f_{01}}{\psi}{\wxu} - a_1\liedede{f_0 + f_1}{f_0}{\psi}{\wxu}=\frac{d^2}{dt^2}\Phi^-(\wla(t))|_{t=\wtau_1^-} &>0 \label{eq: reg switch 1} \\
 F_{001}(\well_2) -a_2 \liededo{f_0}{\psi}{\wx_2}=\frac{d^2}{dt^2}\Phi^-(\wla(t))|_{t=\wtau_2^+} 
&<0 \label{eq: reg switch 2} \\
\Rtre= F_{01}(\well_3)+a_3L_{f_0} \psi(\wx_3)=\frac{d}{dt}\Phi^+(\wla(t))|_{t=\wtau_3}  &<0.
\label{eq: reg switch 3}
\end{align}
\end{ass}

A well-known  second order necessary optimality condition concerning singular arcs is given by the Generalised Legendre condition (see for instance 
\cite[Theorem~20.16]{AS04}), which in our context reduces to 
\begin{equation}
F_{101}(\wla(t)) + a_1\big( \liede{f_{01}}{\psi}{\wxi(t)} - \liedede{f_1}{f_0}{\psi}{\wxi(t)}\big) \geq 0.
\end{equation}
We assume that the inequality here above holds in the strict form.
\begin{ass}[Strong generalised Legendre Condition (SGLC)]\label{ass:  sglc}
For all $t\in[\wtau_1,\wtau_2]$ 
\begin{equation}\label{SL}
F_{101}(\wla(t)) + a_1\big( \liede{f_{01}}{\psi}{\wxi(t)} - \liedede{f_1}{f_0}{\psi}{\wxi(t)}\big)>0 . 
\end{equation} 
\end{ass}
For the purpose of future computations, we introduce the following notation:
\begin{equation}
\mathbb{L}(\ell) =F_{101}(\ell) + a_1\big( \liede{f_{01}}{\psi}{\pi\ell} - \liedede{f_1}{f_0}{\psi}{\pi\ell}\big) \qquad \ell\in T^*M,
\end{equation}
so that equation~\eqref{SL} reads $\mathbb{L}(\wla(t))>0,\ t\in[\wtau_1,\wtau_2]$.

Assumption \ref{ass:  sglc} yields some geometric properties of two subsets 
of $T^*M$ which are crucial for our construction:
\begin{gather}
\Sigma^- =\{\ell\in T^*M  \colon  \Phi^-(\ell)=0\}, \label{Sigma} \\
S^- =\{\ell\in\Sigma^-  \colon  F_{01}(\ell)-L_{f_0}\psi(\pi\ell)=0 \}. \label{esse}
\end{gather}
Indeed, thanks to Assumption~\ref{ass:  sglc}, it is easy to see that, in a neighbourhood of  $\wla([\wtau_1,\wtau_2])$, $\Sigma^-$ is a codimension one 
embedded submanifold of $T^*M$. 
Moreover, Assumption~\ref{ass:  sglc} implies that $\vPhi{-}$ is not tangent to $S^-$, so that $S^-$ is a codimension one embedded submanifold of $\Sigma^-$.
We finally notice that, for every $\ell\in S^-$, the tangent space to $\Sigma^-$ at $\ell$ splits in the following direct sum
\begin{equation} \label{eq: tangent split}
T_{\ell}\Sigma^- =T_{\ell}S^-\oplus \R \vPhi-(\ell).
\end{equation}

\medskip
We end the section with a technical assumption concerning again the zeroes of the running cost along the reference trajectory. Indeed, 
as already stressed, the  flow generated by the maximised Hamiltonian, appearing in \eqref{PMP equation}-\eqref{PMP max},  depends on how many times  the function $\psi$ changes sign along the reference trajectory, as each of these points bears a non-smoothness point. 
This issue has been accurately treated in \cite{CP19}, and the same computations carried out there could extend with no modifications to the current problem. 
Thus, due to the complexity introduced by the presence of a singular arc, and in order to simplify the presentation, we make the following assumption. 
\begin{ass} \label{ass:  segni}
	Along the first bang-arc,  the singular and the last bang arc, the function $\psi$ is positive.
\end{ass}

\begin{remark}
We stress that this assumption does not cause any loss of generality;  the result (Theorem~\ref{th: main result}) holds true also if we drop it, provided that the other assumptions are satisfied and that the second variation
and the maximised flow are suitably computed, according to the rules given in \cite{CP19}. 	
\end{remark}

\section{The second variation}\label{sec:extsecvar}
Following the approach initiated in \cite{ASZ98a}, we write the second variation as an accessory problem,
that is, an LQ optimal control problem defined on the
tangent space to $M$ at $\wxu$.\footnote{ In principle, the basepoint could be any point of the reference trajectory. The choice of $\wxu$ considerably
simplifies the expression of the second variation, as it permits to neglect variations along the first bang arc.}
The admissible control functions of the accessory problem are the admissible control variations of the original optimal control problem, that is, all 
functions $\delta u\in L^{\infty}([0,T],\mathbb{R})$ such that $\wu+\delta u$ is still an admissible control for  problem \eqref{eq: contr sys};
the set of admissible control variations is then completed as a suitable subspace of the Hilbert space $L^2([0,T],\mathbb{R})$ 
(see \cite[Remark~6]{ASZ98a}).
However, facing the problem from the most general point of view, that is, considering all possible admissible variations, is not only cumbersome, 
but possibly pointless: in many general cases, indeed, the space of admissible variations is ``too big'', so that the second variation cannot
be coercive on it (see for instance \cite{PS12} for bang-singular concatenations and \cite{CS16} for an example in the case of totally singular extremals). 
On the other hand, in many cases (\cite{ASZ02b,PS11,CS16}) it has been proved that it is possible to properly reduce 
the set of admissible variations, and still  obtain sufficient conditions for optimality  in terms of the coerciveness of the second variation. 
The goal is then to find the ``smallest'' space of admissible variations such that the coerciveness of the second
variation on it still implies that the Hamiltonian flow is invertible.

In particular, in \cite{ASZ02b} it has been shown that, for bang-bang extremals, it is sufficient to 
consider only the variations of the switching times. As it will be proved in the paper, it turns out that, for problem~\eqref{eq: cost}-\eqref{eq: contr sys}, an appropriate  set of admissible variations is constituted by the variation of the third switching time and the variation of the control function along the singular arc. 

 Clearly, this reduction  considerably simplifies the expression of the second variation. 
The rest of this section is devoted to its construction.
 
\smallskip
It is well known that computing higher order derivatives on manifolds is a delicate task, as they are not invariant under change of coordinates. To overcome this problem
and obtain an intrinsic expression of the second variation, we transform the original problem into  a Mayer one and we pull-back the linearisation of the system along the reference trajectory
to the tangent space $T_{\wxu}M$.
Setting
$\wb{\xi}=(\xi^0,\xi)\in \R\times M$, problem~\eqref{eq: cost}-\eqref{eq: contr sys} reads
\begin{equation}
\mbox{minimise }\xi^0(T)-\xi^0(0)
\end{equation}
 among all solutions of the control system
\begin{equation} \label{eq: contr sys ext}
\left\{
\begin{array}{l}
\dot{\xi^0}(t) =|u(t) \psi(\xi(t))|,\\
\dot{\xi}(t) = \left( f_0 +  u f_1 \right) \circ \xi(t),\\ 
\wb{\xi}(0) =(0,\wx_0), \quad \wb{\xi}(T)\in \R\times\{\wx_T\} ,\\
u\in [-1,1].
\end{array} 
\right. 
\end{equation}
It is immediate to see that the covector $\wb{\lambda}=(\lambda^0,\wla)\in \R\times T^*M$, with
$\lambda^0(t) \equiv -1$ satisfies normal PMP. 

The reference flow from time $\wtau_1$ associated with the system \eqref{eq: contr sys ext} is denoted as $\wb{\wS{t}}$, and is given by
\begin{equation} \label{S cap}
\wb{\wS{t}}(c_1,q)=\begin{pmatrix}
                                         \wS{t}^0(c_1,q)\\\wS{t}(q) 
                                         \end{pmatrix}
=
\begin{pmatrix}
                                         c_1+\int_{\wtau_1}^t |\wu(s)\psi(\wS{s}(q))|\ud s
                                         \\\wS{t}(q) 
                                         \end{pmatrix}
=
\begin{pmatrix}
                                         c_1+\int_{\wtau_1}^t |\wu(s)\wpsi{s}(q))|\ud s
                                         \\\wS{t}(q) 
                                         \end{pmatrix}.
                                         \end{equation}
  \begin{remark}\label{re:notazione}
	Notice that $\wS{t}^0$ does depend on $c_1$, while its differential  does not. In what follows, with some abuse of notation, we write $\ud \wS{t}^0(q)$ for the differential of $q \mapsto \wS{t}^0(c_1, \cdot)$ at a point $q$.
\end{remark}  

Consider some  $\tau_3\in(\wtau_2,T)$ and a measurable control function $v\colon [\wtau_1,\wtau_2]\to (0,1)$, and let $\wb \xi$ be the solution of 
\eqref{eq: contr sys ext},
starting from the point $\wb\xi(0)=(0,\wq_0)$ 
associated with the control 
\begin{equation}
u(t)=
\begin{cases}
1 & t\in[0,\wtau_1),\\
v(t) & t \in (\wtau_1,\wtau_2),\\
0 &  t\in (\wtau_2,\tau_3),\\
-1 & t \in (\tau_3,T].
\end{cases}
\end{equation}
We consider the  piecewise-affine reparametrization of time $\varphi \colon [0,T]\to [0,T]$ defined by 
\begin{equation}
\dot{\varphi}(t)=
\begin{cases}
1 & t \in [0,\wtau_2),\\
\frac{\tau_3-\wtau_2}{\wtau_3-\wtau_2} \quad &  t\in (\wtau_2,\wtau_3),\\
\frac{T-\tau_3}{T-\wtau_3} & t \in (\wtau_3,T],
\end{cases}\qquad \varphi(0) = 0,
\end{equation}
and we set
$\wb{\eta}_t:=\begin{pmatrix}
                                 \eta_t^0\\ \eta_t
                                \end{pmatrix}
=\wb{\wS{t}}^{-1} (\wb \xi(\varphi(t)))$. 
Let  $\gamma_0,\gamma_T\colon M \to \R$ be two smooth functions such that
\begin{equation}
\ud\gamma_0(\wh q_0)=\well_0,
\qquad 
\ud\gamma_T(\wh q_T)= - \well_T,
\end{equation} 
and let $\wh \gamma_0:=\gamma_0\circ \wS{0}$, $\wh \gamma_T:=\gamma_T \circ\wS{T}$.
Then, the cost can be written in terms of the pull-back trajectory $\wb{\eta}_t $ as
\begin{equation} \label{eq: costo J}
 J(u)= \wS{T}^0(\wb \eta_T) - \wS{0}^0(\wb \eta_0) +\wh \gamma_0(\eta_0)+\wh \gamma_T(\eta_T).
\end{equation}
Thanks to PMP, it is possible to show that the first
variation of $J$ evaluated at $\wu$ is null (see Appendix~\ref{app secvar} for more details).

In order to compute the second variation of $J$, we introduce the pullbacks to time $\wtau_1$ of the vector fields governing the dynamics of $\wb{\eta}_t$:
\begin{equation}
g^1_t :=\wS{t*}^{-1}f_1\circ\wS{t}, \quad
k_3 := \wS{\wtau_3*}^{-1}f_0\circ\wS{\wtau_3}, \quad
k_4 := \wS{\wtau_3*}^{-1}(f_0-f_1)\circ\wS{\wtau_3}, \quad
k=k_4-k_3:=-\wS{\wtau_3*}^{-1}f_1\circ\wS{\wtau_3}.
\end{equation}
Setting $\ep := -(\tau_3 - \wtau_3)$, the second variation of the optimal control problem is  given by
\begin{equation}
\begin{split}\label{eq: secvar 1}
J\se[\delta v, \ep]^2 & := 
 \int_{\wtau_1}^{\wtau_2} \delta v(s) 
\lieder{\delta \eta(s)}{\big(\wpsi{s} + L_{g^1_s} \big(\wh\gamma_T + \wS{T}^0 - \wS{s}^0\big)\big)}
\ud s \\
 & - \dfrac{\ep^2}{2}\liederdo{k}{ \big(\wh\gamma_T + \wS{T}^0 - \wS{\wtau_2}^0\big)}  
+ \dfrac{\ep^2}{2}
\lieder{\liebr{k_4}{k_3}}{\big(\wh\gamma_T + \wS{T}^0 - \wS{\wtau_2}^0\big) }
\\
& -\dfrac{\ep^2}{2}\lieder{k_4}{\wpsi{\wtau_3}}, 
\end{split}
\end{equation}
where $\delta v \in L^\infty([\wtau_1, \wtau_2])$ and  $\delta \eta_t$ is the linearisation of $\eta_t$ at $ u = \wh u$, and satisfies the control system
\begin{equation} \label{eq: delta eta}
\dot{\delta \eta}_t=
\begin{cases}
0 & t\in[0,\wtau_1),  \\
\delta v(t) g_t^1 (\wxu) \quad 
& t\in(\wtau_1,\wtau_2),  \\
-\frac{\ep}{\wtau_3 - \wtau_2} k_3(\wxu) & t\in(\wtau_2,\wtau_3),  \\
\frac{\ep}{T - \wtau_3} k_4(\wxu) & t\in(\wtau_3,T],
\end{cases} 
\qquad
\begin{matrix}
\delta\eta_0= 0, \\ 
\delta \eta_T=0. 
\end{matrix}
\end{equation}

The second variation \eqref{eq: secvar 1} is degenerate, as the quadratic term in $\delta v$ (Legendre
term) is missing; to overcome this issue, we perform a Goh transformation, that is, we integrate the control variation and we add an additional variation $\ep_0$,
in the same spirit of \cite{PS11}:
\begin{equation} \label{eq: goh}
w(t)  := \int_t^{\wtau_2} \delta v(s) \ud s, \quad t \in [\wtau_1, \wtau_2], \qquad
\ep_0  := w(\wtau_1).
\end{equation}
We thus obtain the {\em extended admissible variations}  as the triples
$ \de := (\ep_0,\ep,w)\in \R\times \R\times L^2([\wtau_1,\wtau_2])$
such that the system
\begin{equation}\label{eq:zetaridotto}
\begin{cases}
 \dot\zeta(t) = w(t)\dot g^1_t(\wxu), \\ 
 \zeta(\wtau_1) =  \ep_0 f_1(\wxu),  \\  
\zeta(\wtau_2) =  -  \ep k(\wxu)
\end{cases}
\end{equation}
admits a solution, and the extended second variation as the quadratic form
\begin{equation}
\begin{split}\label{eq: secvar ext}
J\se_{e}[\de]^2 &= 
\dfrac{\ep_0^2}{2} 
\lieder{f_1}{\big(\psi  +  L_{f_1} \big( \wh\gamma_T + \wS{T}^0\big)\big)} 
+ \dfrac{1}{2}\int_{\wtau_1}^{\wtau_2} \left( w^2(s) R(s) + 2w(s) L_{\zeta(s)} A(s,\wxu) \right) \ud s
\\ 
& -\dfrac{\ep^2}{2}L_{k_4} \wpsi{\wtau_3} (\wxu)  
- \dfrac{\ep^2}{2}\liederdo{k}{ \big(\wh\gamma_T + \wS{T}^0 - \wS{\wtau_2}^0\big)}  
+ \dfrac{\ep^2}{2}
\lieder{[k_4, k_3]}{\big(\wh\gamma_T + \wS{T}^0 - \wS{\wtau_2}^0\big)},
\end{split}
\end{equation}
where we put
\begin{align}
A(s,q)  := &\AA{s}{q}, \\
R(s)  := &\RR{s}. 
\end{align}
\begin{remark}
We remark that $R(s)=\mathbb{L}(\wla(s))$, which is positive by Assumption~\ref{ass:  sglc}. 	
\end{remark}

We can now state our final assumption.
\begin{ass} \label{ass:  coerc}
The quadratic form $J\se_e$ \eqref{eq: secvar ext} is coercive on the space of admissible variations 
	\begin{equation} 
	\cW  := \big\{
	\delta e = (\ep_0,  \ep, w) \in \R\times \R \times  
	 L^2([\wtau_1, \wtau_2])
	\text{ such that system \eqref{eq:zetaridotto} admits a solution}
	\big\}.
	\end{equation}
\end{ass}

\subsection{Consequences of coerciveness of $J\se_{e}$}


\begin{lemma}\label{le:f1}
Assume that Assumption~\ref{ass:  coerc} holds true. Then $f_1(\wxu)\neq 0$.
\end{lemma}
\begin{proof}
	Assume by contradiction that $f_1(\wxu) = 0$.
	Then $\de  := \begin{pmatrix}
  \ep_0 = 1, \ep = 0, w \equiv 0 
	\end{pmatrix}$
	is a non-trivial admissible variation and
	$J\se_e[\de]^2 =\lieder{f_1}{\big(\psi  +  L_{f_1} \big( \wh\gamma_T + \wS{T}^0\big)\big)} =0$, since $f_1(\wxu) = 0$.  Thus we have a contradiction. 
\end{proof}

\begin{lemma}\label{le:alpha}
	Under Assumption~\ref{ass:  coerc}, there 
exist a neighbourhood $U_{\wxu}$ of $\wxu$ in $M$ and a smooth function $\alpha \colon U_{\wxu}\to \R$ such that
	\begin{equation}  \label{eq: Jse}
	\begin{split}
	J\se_{e}[\de]^2 &= \dfrac{\ep_0^2}{2} 
	\liederdo{f_1}{\big( \alpha+ \wh\gamma_T + \wS{T}^0\big)}
	+ \dfrac{1}{2}\int_{\wtau_1}^{\wtau_2} \left( w^2(s) R(s) + 2w(s) L_{\zeta(s)} A(s, \wxu) \right) \ud s 
	\\
	&- \dfrac{\ep^2}{2}  \lieder{k_4}{\wpsi{\wtau_3}}  - \dfrac{\ep^2}{2}\liederdo{k}{ \big(\wh\gamma_T + \wS{T}^0 - \wS{\wtau_2}^0\big)}  
	+ \dfrac{\ep^2}{2}
	\lieder{\liebr{k_4}{k_3}}{\big(\wh\gamma_T + \wS{T}^0 - \wS{\wtau_2}^0\big)}	.
	\end{split}
	\end{equation}

	In particular, it holds
	\begin{equation}
	\liede{f_1}{\alpha}{q} = \psi(q) \quad \forall q  \in U_{\wxu}, \qquad \ud\alpha(\wxu) = \wh\ell_1.
	\end{equation}
\end{lemma}
\begin{proof}
Thanks to Lemma \ref{le:f1}, we can choose local coordinates $(y_1, y_2, \ldots, y_n)$ around $\wxu$ such that $f_1 \equiv \dfrac{\partial}{\partial y_1}$ in a neighbourhood $U_{\wxu}$ of $\wxu$.

Let $\wh\ell_1 = \begin{pmatrix}
\wh p_1, \wh p_2, \ldots, \wh p_n
\end{pmatrix}$. Thus $\wh p_1 = F_1(\wh\ell_1) = \psi(\wxu)$.
Possibly shrinking $U_{\wxu}$, let $\alpha \colon U_{\wxu}  \to \R$ be the solution of the Cauchy problem
\begin{equation}
\begin{cases}
\dfrac{\partial \alpha}{\partial y_1}(y_1, y_2, \ldots, y_n) = \psi(y_1, y_2, \ldots, y_n), \\
\alpha (0, y_2, \ldots, y_n) = \sum_{j=2}^n \wh p_j y_j.
\end{cases}
\end{equation}
Then, by construction, $\ud\alpha(\wxu) = \begin{pmatrix}
\psi(\wxu), \wh p_2, \ldots, \wh p_n
\end{pmatrix} = \wh\ell_1$. 
The term  $L_{f_1}\psi(\wxu)$ in $J\se_e$  can thus be replaced by 
$\liederdo{f_1}{\alpha}$, and this proves  the claim.
\end{proof}

We claim that 
the term $\liederdo{f_1}{\big( \alpha+ \wh\gamma_T + \wS{T}^0\big)}$ in the expression for $J\se_e$ can be replaced by $\uD^2\big(\alpha+\wh\gamma_T+ \wS{T}^0\big)(\wxu)[f_1(\wxu)]^2$.
Indeed, by definition of $\gamma_T$, it holds $\ud\gamma_T(\wx_T) = -\well_T$; since, by \eqref{eq: solution lambda},
we have that $\well_T= \big(\wh\ell_1 + \ud \wS{T}^0\big)\wSinv{T*}$, we obtain that
$\ud\alpha(\wxu)=\well_1=- \big(\ud\wh\gamma_T+  \ud \wS{T}^0\big)(\wxu)$, and we are done. We can thus write 
$J\se_e$ on $\cW$ as 
\begin{equation}
\begin{split} \label{eq: Jse II}
J_e\se[\de]^2 & =  \dfrac{1}{2}\uD^2\big(\alpha + \wh\gamma_T + \wS{T}^0 \big) (\wxu)[\ep_0 f_1(\wxu)]^2 
+ \dfrac{1}{2}\int_{\wtau_1}^{\wtau_2} \left( w^2(s) R(s) + 2w(s) L_{\zeta(s)} A(s, \wxu) \right) \ud s \\
& -\dfrac{\ep^2}{2}L_{k_4} \wpsi{\wtau_3} (\wxu)  
 - \dfrac{\ep^2}{2}\liederdo{k}{ \left(\wh\gamma_T + \wS{T}^0 - \wS{\wtau_2}^0\right)}  
+ \dfrac{\ep^2}{2}
\lieder{\liebr{k_4}{k_3}}{\big(\wh\gamma_T + \wS{T}^0 - \wS{\wtau_2}^0\big)}
.
\end{split}
\end{equation}

We now extend the space of admissible variations, in the following way: we remove the constraint on $\zeta(\wtau_1)$ and we consider 
the control system 
\begin{equation}\label{eq: variazioni ultime}
\begin{cases}
\dot\zeta(s) = w(s) \dot g^1_s(\wxu), \\ \zeta(\wtau_1) = \dz \in T_{\wxu} M, \\ \zeta(\wtau_2) = - \ep k(\wxu).	
\end{cases}
\end{equation}
The associated space of variations is thus defined as
	\begin{equation} 
\bcW  := \big\{
\delta e = (\delta z,  \ep, w) \in T_{\wxu} M \times \R \times  
L^2([\wtau_1, \wtau_2])
\text{ such that system \eqref{eq: variazioni ultime} admits a solution}
\big\}.
\end{equation}
Applying \cite[Theorem~11.6]{Hes66}, we can easily verify that \eqref{eq: Jse II} defines a Legendre
form on $\bcW$. 
For $s\in \R^+$, we define $\theta(y_1, y_2, \ldots, y_n)  := \dfrac{s}{2}\sum_{j=2}^n y_j^2$ where $(y_1, y_2, \ldots y_n)$ are the coordinates defined in the proof of Lemma \ref{le:alpha}. 
Applying \cite[Theorem 13.2]{Hes66},  we obtain that, under Assumption \ref{ass:  coerc}  and if $s$ is large enough, then the quadratic form
\begin{equation}
\bJ [\de]^2 := J\se_e[\de]^2 + \dfrac{1}{2}\uD^2\theta[\dz]^2
\end{equation} 
is coercive on   $\bcW$.

We can write $\bJ$ explicitly as
\begin{equation} \label{eq: Jssecondo}
\begin{split}
\bJ[\de]^2 & =  \dfrac{1}{2}\uD^2\big(\alpha +\theta+ \wh\gamma_T + \wS{T}^0 \big) (\wxu)[\dz]^2 
+ \dfrac{1}{2}\int_{\wtau_1}^{\wtau_2} \left( w^2(s) R(s) + 2w(s) L_{\zeta(s)} A(s,\wxu) \right) \ud s \\
& -\dfrac{\ep^2}{2}L_{k_4} \wpsi{\wtau_3} (\wxu)   - \dfrac{\ep^2}{2}\liederdo{k}{ \big(\wh\gamma_T + \wS{T}^0 - \wS{\wtau_2}^0\big)}  
+ \dfrac{\ep^2}{2}L_{[k_4,k_3]} \big(\wh\gamma_T + \wS{T}^0 - \wS{\wtau_2}^0\big)(\wxu) .
\end{split}
\end{equation}

\section{Construction of the over-maximised flow}\label{sec:overmaxi}
As briefly mentioned in the Introduction, Hamiltonian methods to prove the optimality of an extremal need two ingredients:
a Lagrangian submanifold $\Lambda_1$ of the cotangent bundle, containing a point of the reference extremal, and a over-maximised Hamiltonian flow (that is, the flow associated with a 
Hamiltonian function which is greater than or equal to $H_{\max}$, and  which  coincides with it along the reference extremal, 
at least up to the first order); the optimality of the reference extremal is proved if, at each time $t\in[0,T]$, the image of $\Lambda_1$ 
under the flow of the over-maximised Hamiltonian projects 
 one-to-one onto a neighbourhood of $\wxi(t)$ in the base manifold (see Section~\ref{sec:  main results}).
 
This section is devoted to the construction of the over-maximised Hamiltonian and of its associated flow, that we refer to
as \emph{over-maximised flow}. It is defined patching together some piecewise smooth Hamiltonian flows, each one defined on a suitable neighbourhood
of each arc.  We stress that, to achieve this task, to coerciveness of $J_e\se$ is not needed; though, we extensively use
Assumptions~\ref{ass:  nonzero at switch}--\ref{ass:  segni}, which are assumed to hold true throughout the
whole section.

\subsection{The over-maximised Hamiltonian near the singular arc}\label{sec:  over max Ham}

By definition, the singular arc evolves on the hypersurface $\Sigma^-=\{\Phi^-=0\}$ and $\Phi^+(\wla(t)) > 0$. Thus, in a sufficiently small neighbourhood of 
$\wla([\wtau_1,\wtau_2])$, the maximised Hamiltonian is given by
\begin{equation}
H_{\max}=\begin{cases}
F_0+\Phi^- \quad& \mbox{if } \Phi^-\geq 0,\\
 F_0   
\quad& \mbox{if } \Phi^-\leq 0.
\end{cases}
\end{equation}
Therefore, $H_{\max}$ is continuous, but its associated Hamiltonian vector field is not well defined on $\Sigma^-$: for every smooth Hamiltonian $v(t, \ell)$,  every 
Hamiltonian of the form $F_0+ v \Phi^-$  coincides with $H_{\max}$ on $\Sigma^-$.  

On the other hand, no Hamiltonian of the form $F_0+ v \Phi^-$ can be used in place of the maximised Hamiltonian: 
indeed, by Assumption~\ref{ass:  sglc}, for every $t \in(\wtau_1, \wtau_2)$ and every  $\cU$ neighbourhood of $\wla(t)$ in $\Sigma^-$, there exists some $\ell \in \cU$ such that $F_{01}(\ell) - L_{f_0}\psi(\pi\ell)< 0 $. For any choice of $v > 0$, the flow of $\vF{0} + v \vPhi{-}$ sends $\ell$ in the region $\{\Phi^- < 0\}$, where 
$F_0 + v \Phi^-$ is no more the maximised Hamiltonian.  
 Indeed, whatever the choice of $v$, the Hamiltonian vector field $\vF{0} + v \vPhi{-}$ is tangent to $\Sigma$ only on $ S^- =\{\ell\in\Sigma^-  \colon  F_{01}(\ell)-L_{f_0}\psi(\pi\ell)=0 \}$. 
 We refer to \cite{PS11} for a detailed description of the phenomenon.  

However, as proposed in \cite{Ste07}, the flow associated with $H_{\max}$ may be replaced 
by the flow of a suitable over-maximised Hamiltonian tangent to $\Sigma^-$, at least for $t\in[\wtau_1,\wtau_2]$.
The construction of such a flow in a neighbourhood of $\wla([\wtau_1,\wtau_2])$ in $\Sigma^-$ relies
on the definition of a $C^1$ over-maximised Hamiltonian  that agrees with the maximised one 
 at least up to the first order along the reference extremal, and whose Hamiltonian vector field is tangent to $\Sigma^-$ (see \cite{Ste07,PS11,SZ16,CS16} for  similar constructions).

In order to do so, following the steps of \cite{CS10,CS16}, we substitute $F_0$ with a suitable Hamiltonian $H_0$ which is constant along the integral lines of $\Phi^-$.  

\begin{lemma}\label{le:theta}
 There exist a neighbourhood $\cU$ of $\wla ([\wtau_1,\wtau_2])$ in $T^*M$ and a smooth function 
 $\vartheta \colon \cU \to \R$ such that
 \begin{equation}
\big(F_{01}-L_{f_0}\psi\circ\pi\big)\circ \exp\big(\vartheta(\ell)\vPhi{-}\big)(\ell)=0 \qquad \forall \ell\in \cU,
 \end{equation}
and, for any $\ell \in \cU \cap S^-$, 
\begin{equation}\label{dtheta}
\scal{\ud\vartheta(\ell)}{(\cdot)} =  - \frac{1}{\mathbb{L}(\ell)}\left(
\scal{\ud F_{01}(\ell)}{(\cdot)} - 
\liedede{\pi_*(\cdot)}{f_0}{\psi}{\pi\ell}
\right).
\end{equation}
\end{lemma}

\begin{proof}
By Assumption~\ref{ass:  sglc}, we can apply the implicit function theorem to the function 
$(s,\ell)\mapsto\big(F_{01} - L_{f_0}\psi\circ\pi\big)\circ \exp\big(s\vPhi{-}\big)(\ell)$
at the point $(0,\wla(t))$, and obtain the result.
\end{proof}

Thanks to Lemma~\ref{le:theta}, we can define the following Hamiltonian $\wF_0 \colon \cU\to \R$
\begin{equation}
\wF_0(\ell)=F_0\circ\exp(\vartheta(\ell)\vPhi{-})(\ell). 
\end{equation}

\begin{proposition} \label{pro: H0}
The Hamiltonian $H_0$ satisfies the following properties.
\begin{enumerate}
\item For every $\ell \in \Sigma^- \cap \cU$, $\vec{H}_0(\ell)$ is tangent to $\Sigma^-$ and is given by 
\begin{equation} \label{eq: wF}
\vec{\wF}_0(\ell)=\exp(- s\vPhi{-})_*\vec{F}_0 \circ \exp(s\vPhi{-})(\ell)\vert_{s =\vartheta(\ell)}.
\end{equation}
In particular,  it coincides with $\vec{F}_0(\ell)$ if $ \ell \in S^-\cap \cU$.
\item Possibly shrinking $\cU$, $H_0 (\ell) \geq F_0(\ell)$ for any $\ell \in \Sigma^- \cap \cU$. Equality holds if and only if $\ell \in S^-$.
\item For every smooth function $\upsilon (t,\ell) \colon  \R\times T^*M \to \R$, 
 the Hamiltonian vector field associated with $H_0(\ell)+\upsilon(t,\ell) \Phi^-(\ell)$ is tangent to $\Sigma^-$.  
\end{enumerate}
\end{proposition}

\begin{proof}
Claim (1)~easily follows from the fact that $\scal{\ud F_0}{\vPhi-} = 0$ on $S^-$. Indeed, by construction, $\wF_0$ is constant along the integral lines of $\vPhi{-}$, so that $\scal{ \ud\Phi^-}{\vec\wF_0} = \scal{\ud H_0}{\vPhi{-}}=0$, for any $\ell\in \cU$,
i.e.~$\vwF_0$ is tangent to $\Sigma^-$. Equation~\eqref{eq: wF} can be verified with simple computations.

 Since  $\vartheta=0$ on $S^-$, by definition and equation~\eqref{eq: wF} we have that $\wF_0=F_0$ and $\vec\wF_0=\vec F_0$ on $S^-$.
Thus the differential of $\wF_0-F_0$ is identically zero on $S^-$, and, for any $\ell \in S^- \cap \cU$,  we can compute its second derivative:
\begin{align*}
\uD^2(\wF_0-F_0)(\ell)[\dl]^2&=
2 \scal{\ud\,(L_{\vPhi{-}} F_0)(\ell)}{\dl} 
\scal{\ud\vartheta(\ell)}{ \dl} + L^2_{\vPhi{-}} F_0(\ell)
\scal{\ud\vartheta(\ell)}{ \dl}^2\\
&= -L^2_{\vPhi{-}} F_0(\ell) \scal{\ud\vartheta(\ell)}{ \dl}^2.
\end{align*}
Noticing that $L^2_{\vPhi{-}} F_0(\ell) = - 
\mathbb{L}(\ell)$ and thanks to Assumption~\ref{ass:  sglc}, we see that the expression here above is non negative and it  vanishes only 
if $\scal{\ud\vartheta(\ell)}{ \dl}=0$;
equation~\eqref{eq: tangent split} and the fact that $\scal{ \ud\vartheta(\ell)}{\vPhi{-}(\ell)} =-1 $ for $\ell\in  \Sigma^-\cap \mathcal{U}$ prove Claim (2).

Since $\vec{H}_0$ is tangent to $\Sigma^-$, then $\vec{H}_0+u\vPhi-$ is tangent to $\Sigma^-$
too, for every $u\in \R$. Let $\upsilon$ be a smooth function on $\R\times T^*M$. Then the Hamiltonian 
field associated with $H_0+\upsilon(t,\ell)\Phi^-$ is given by 
\begin{equation}
\vec{H}_0(\ell)+\upsilon(t,\ell)\vPhi-+\Phi^-(\ell) \vec{\upsilon}(t,\ell),
\end{equation}
and, by definition of $\Sigma^-$, this completes the proof. 
\end{proof}

\subsection{The over-maximised flow}

Thanks to the regularity of the bang  and zero arcs, for any $t\in [0,\wtau_1) \cup (\wtau_2,\wtau_3)\cup(\wtau_3,T]$, it is possible to find a neighbourhood of $\wla(t)$ where $H_{\max}$  and its associated vector field are unambiguously defined.
As observed in the previous section, this is no longer true for $t\in [\wtau_1,\wtau_2]$. 
On the other hand, thanks to Proposition~\ref{pro: H0}, we know that any over-maximised Hamiltonian of the form $H_0+v\Phi^-$ may replace the maximised Hamiltonian in a neighbourhood of the singular arc in $\Sigma^-$.

Here below we show how to concatenate the flows of $H_{\max}$ 
and of the over-maximised Hamiltonian, in order to obtain a flow defined for all $t\in[0,T]$.

\smallskip
\paragraph{The first bang arc.}

We first construct the over-maximised flow for $t\in[0,\wtau_1]$. 

Proposition~\ref{pro:  H0} guarantees that $H_0\geq F_0$ in a neighbourhood of $\wla([\wtau_1,\wtau_2])$ contained in $\Sigma^-$ only. In other words,
if we want to use $H_0$ to construct the over-maximised Hamiltonian, we have to be sure that, for $t\in[\wtau_1,\wtau_2]$, the image of the sub-manifold
$\Lambda_1$ under the over-maximised flow is in $\Sigma^-$. 
For this reason, it is convenient to start from $\wtau_1$ and construct the flow integrating backward in time.

Fix some $\epsilon>0$ and consider a sufficiently small tubular neighbourhood $\eulU$ of $\wla((\wtau_1-\epsilon,\wtau_1+\epsilon))$ in $T^*M$.
The manifold $\Sigma^-$ separates $\eulU$ in two regions, one in which $\Phi^->0$ (and $H_{\max}=F_0+\Phi^-$), the other one in which $\Phi^-<0<\Phi^+$
(and $H_{\max}=F_0$), since $\psi(q)\neq 0$ in a neighbourhood of $\wxi(\wtau_1)$; in particular, the first bang arc is 
contained in the first region. On the other hand, by Assumption~\ref{ass:  sglc}, the manifold $S^-$ separates $\eulU\cap\Sigma^-$ 
into two regions, in which $F_{01}-L_{f_0} \psi\circ\pi$
has different sign. 
Consider now a small neighbourhood of $\well_1$ in $\Sigma^-$. The trajectories obtained by integrating backward in time the flow generated 
by $\vec{F}_0+\vPhi-$, starting at $t=\wtau_1$ from a point  $\ell$ satisfying $F_{01}(\ell)-L_{f_0} \psi(\pi\ell) < 0$,  immediately leave $\Sigma^-$ and enter
in the region $\{\Phi^->0\}$; in particular, they evolve with the maximised flow and stay close to the reference 
extremal (if $\ell$ is sufficiently close to $\well_1$). 
The same happens for trajectories starting from a point $\ell$ such that $F_{01}(\ell)-L_{f_0} \psi(\pi\ell) = 0$, thanks to Assumption~\ref{ass:  reg switch}.

On the contrary, the integral curves of 
 $\vec{F}_0+\vPhi-$ with an initial condition $\ell$ satisfying $F_{01}(\ell)-L_{f_0} \psi(\pi\ell) > 0$  immediately enter into
 the region $\{\Phi^-<0\}$, so that they are not integral curves of $\vec{H}_{\max}$ and may soon leave $\eulU$.
To fix this issue, for initial conditions belonging to the region where $F_{01}(\ell)-L_{f_0} \psi(\pi\ell)  > 0$, we substitute the flow of $\vec{F}_0+\vPhi-$ with the one of $\vec{H}_0+\vPhi-$, until the trajectories  reach $S^-$.
This construction is explained in Proposition~\ref{pro:  primo arco} here below, whose proof relies on the following lemma.

\begin{lemma} \label{lemma:  t1}
There exist a neighbourhood $\cO_1$ of $\well_1$ in $\Sigma^-$ and a smooth function $t_1 \colon \cO_1\to \R$ satisfying $t_1(\well_1)=\wtau_1$ such that
\begin{equation}
(F_{01} -(L_{f_0} \psi)\circ\pi)\circ \exp\big( (t_1(\ell)-\wtau_1)(\vec{H}_0+\vPhi-)\big)(\ell)=0 \quad \forall \ell\in \cO_1.
\end{equation}
Moreover, $t_1(\ell)\gtreqless \wtau_1$ if and only if $F_{01}(\ell)-L_{f_0}\psi(\pi \ell)\lesseqgtr 0$.
\end{lemma}

\begin{proof}
The existence of the function $t_1$ is  a  straightforward  application of the implicit function theorem to the function
$\varphi(t,\ell)= (F_{01} -(L_{f_0} \psi)\circ\pi)\circ\exp((t-\wtau_1)(\vec{H}_0+\vPhi-))(\ell)$ at $(\wtau_1,\well_1)$, which is possible since
$\frac{\partial }{\partial t}\varphi(t,\ell)|_{(\wtau_1,\well_1)}>0$,
by
Assumption~\ref{ass:  reg switch}. 

The sign of $t_1(\ell)-\wtau_1$ is determined by the fact that $t_1(\ell)=\wtau_1$ for every $\ell\in S^-\cap \cO_1$, and again by Assumption~\ref{ass:  reg switch}. 
\end{proof}

We define the piecewise smooth function $\tau_1 \colon  \cO_1\to \R$ as
\begin{equation} \label{eq: wtau1}
\tau_1(\ell)=\min\{ t_1(\ell),\wtau_1\},
\end{equation}
and the flow $\cH_1  \colon  [0,\wtau_1]\times  \cO_1 \to T^*M$ as
\begin{equation} \label{eq: flusso primo arco}
\cH_1(t,\ell)=
\begin{cases}
\exp\big((t-\wtau_1) (\vec{H}_0+\vPhi-)\big) (\ell) & t\in[\tau_1(\ell),\wtau_1],\\
\exp\big((t-\tau_1(\ell))(\vec{F}_0+\vPhi-)\big)\circ\exp\big((\tau_1(\ell)-\wtau_1) (\vec{H}_0+\vPhi-	)\big) (\ell) \quad & t\in[0,\tau_1(\ell)).
\end{cases}
\end{equation}
\begin{remark}
 Clearly, if  $t_1(\bar\ell)\geq \wtau_1$ for some $\bar\ell$, then 
 $\cH_1(t,\bar\ell)$ is the flow of $\vec{F}_0+\vPhi-$ for every $t\in[0,\wtau_1]$.
\end{remark}

\begin{proposition} \label{pro: primo arco}
The flow $\cH_1$ defined above is $C^1$. Moreover
\begin{equation} \label{eq: phi primo arco}
\Phi^-(\cH_1(t,\ell))=0 \quad \forall t\in[\tau_1(\ell),\wtau_1], \qquad 
\Phi^-(\cH_1(t,\ell))>0 \quad \forall t\in[0,\tau_1(\ell)).
\end{equation}
In particular 
\begin{equation} \label{eq: flusso sotto 1}
\big( \pi \cH_1(t,\cdot)\big)_*|_{\well_1}=\wh{S}_{t*}\pi_* \quad \forall t\leq \wtau_1. 
\end{equation}

\end{proposition}

\begin{proof}
At every point $(t,\ell)$ such that $t\neq \tau_1(\ell)$, the flow is well defined and smooth. Therefore,	
to prove its regularity on the whole $[0,\wtau_1]\times \cO_1$, it suffices to verify the continuity of its derivatives at points of the form $(t,\ell)=(\tau_1(\bar\ell),\bar\ell)$.
In particular, we can distinguish two cases, that is, $t_1({\bar\ell})>\wtau_1$ and $t_1(\bar\ell)\leq \wtau_1$. In the former case, the flow coincides with 
$\exp((t-\wtau_1) (\vec{F}_0+\vPhi-))(\ell)$ for every $t\in[0,\wtau_1]$ and for every $\ell$ in a neighbourhood of $\bar{\ell}$ in $\Sigma^-$, thus it is $C^1$.

If instead $\tau_1(\bar\ell)\leq\wtau_1$, then the flow starting from points $\ell$ in a neighbourhood of $\bar{\ell}$  has different expressions according to the sign of $t-\tau_1(\ell)$. However, by straightforward computations, it is
easy to prove that they coincide as $(t,\ell)\to(\tau_1(\bar\ell),\bar\ell)$, for every $\bar \ell\in \cO$  and that the first order partial derivatives are continuous.

Let us now prove equation~\eqref{eq: phi primo arco}. First of all, we recall that $\frac{\partial}{\partial t}\Phi^-(\cH_1(t,\ell))|_{t=\tau_1(\ell)}=
(F_{01} -(L_{f_0} \psi)\circ\pi)
)\circ\cH_1(t,\ell)|_{t=\tau_1(\ell)}$.

If $t_1(\ell)>\wtau_1$, then $\tau_1(\ell)=\wtau_1$ and, by Lemma~\ref{lemma:  t1},  
$F_{01} -(L_{f_0} \psi)\circ\pi<0$, so that, if $\ell$ is close enough to $\well_1$,  equation ~\eqref{eq: phi primo arco} follows immediately from a first order Taylor expansion, with respect to the first variable, at $t = \wtau_1$.

Let us now consider the case in which $t_1(\ell)\leq\wtau_1$.
By construction, $\vec{H}_0+\vPhi-$ is tangent to $\Sigma^-$, so that  $\Phi^-(\cH_1(t,\ell))=0$ 
for $ t\in[\tau_1(\ell),\wtau_1]$.
 At $t=\tau_1(\ell),$ $\mathcal{H}_1(t,\ell)$ is in $S^-$, that is, $(F_{01} -(L_{f_0} \psi)\circ\pi)
)\circ\cH_1(t,\ell)|_{t=\tau_1(\ell)}=0$, so that we must look at the second order Taylor expansion
of $t\mapsto \Phi^-(\cH_1(t,\ell))$ at $\wtau_1$.
By Assumption~\ref{ass:  reg switch},  the second order derivative of this map at $(t,\ell)=(\wtau_1,\well_1)$ is strictly positive, so that,  by continuity, it is strictly positive also at $(\tau_1(\ell),\ell)$, for $\ell$ close enough to $\well_1$.

We can conclude that there exists a $\epsilon >0$ and a neighbourhood $\cO_1$ of $\well_1$ in
$\Sigma^-$ such that $\Phi^-(\mathcal{H}_1(t,\ell))>0$ for $t\in(\tau_1(\ell)-\epsilon,\tau_1(\ell))$, for every $\ell\in \mathcal{O}_1$. 
Possibly shrinking $\cO_1$, we can conclude that 
the inequality is satisfied for every $t\in[0,\tau_1(\ell))$.
\end{proof}

\paragraph{The singular arc.}
We recall that for any $t \in [\wtau_1, \wtau_2]$, the reference extremal $\wla$ takes values in $S^-$. Moreover, thanks to Assumption~\ref{ass:  sglc} and since $\frac{\ud^2}{\ud t^2}\Phi^-(\wla(t))=0$ for any $t \in (\wtau_1, \wtau_2)$, the reference control along the singular arc can be computed in a {\em feedback Hamiltonian form}. More precisely
\begin{equation}
\wu(t) = -\frac{F_{001}(\ell) - \liededo{f_0}{\psi}{\pi \ell}}{\mathbb{L}(\ell)} \vert_{\ell = \wla(t)}\qquad \forall t \in (\wtau_1, \wtau_2).
\end{equation}
In a neighbourhood of $\wla([\wtau_1,\wtau_2])$ in $S^-$, we thus define 
\begin{equation} \label{feedback}
\nu(\ell)=-\frac{F_{001}(\ell) - \liededo{f_0}{\psi}{\pi \ell}}{\mathbb{L}(\ell)}.
\end{equation}
We extend $\nu$ to a neighbourhood of $\wla([\wtau_1,\wtau_2])$ in $\Sigma^-$ by setting it constant along the integral lines of 
$\vPhi-$, and then to a full-measure neighbourhood of the range of the singular arc by setting it constant along the integral lines of the Hamiltonian field associated with 
$F_{01}-L_{f_0}\psi$~\footnote{Indeed, thanks to Assumption \ref{ass:  sglc},  $T_{\ell}(T^*M)=T_{\ell}\Sigma^-\oplus 
	\overrightarrow{F_{01}-L_{f_0}\psi\circ \pi}(\ell)$  for every $\ell \in \Sigma^-$ in a neighbourhood of $\wla([\wtau_1, \wtau_2])$.  }.

We set
\begin{equation}
K(\ell)=
H_0(\ell) +\nu(\ell) \Phi^-(\ell) , 
\end{equation}
and 
define the over-maximised flow on the interval $[\wtau_1,\wtau_2]$ as the flow of $\vec{K}$:
\begin{equation}
\cK(t,\ell)=\exp((t-\wtau_1)\vec{K})(\ell). 
\end{equation}

\begin{proposition} \label{pro: K tangente}
The manifolds $\Sigma^-$ and $S^-$ are invariant under the action of the flow of $K$.	
Moreover, $\vPhi-$ is invariant with respect to the flow of $K$ on $\Sigma^-$, that is, for $\ell$ belonging to a small neighbourhood of $\wla([\wtau_1,\tau_2])$ in $\Sigma^-$, it holds
\begin{equation}\label{eq:invarianza}
\cK(t,\ell)_{*}\vPhi-(\ell)=\vPhi-\circ\cK(t,\ell)
\qquad 
\forall t\in[\wtau_1,\wtau_2].
\end{equation}
\end{proposition}
This result is proved (in a more general version) in Proposition \ref{prop: preserva Phi}.

\paragraph{The inactivated arc.}
The construction of the over-maximised flow on a right hand side neighbourhood of $\wtau_2$
presents the same issues as its construction on $[0,\wtau_1]$, thus 
we overcome these difficulties likewise.

\begin{lemma}
Possibly shrinking $\cO_1$ and setting $\cO_2=\cK(\wtau_2,\cO_1)$, there exists a smooth function $t_2 \colon \cO_2\to \R$ satisfying $t_2(\well_2)=\wtau_2$ such that
\begin{equation}
(F_{01} -(L_{f_0} \psi)\circ\pi)\circ\exp((t_2(\ell_2)-\wtau_2)\vec{H}_0)(\ell_2)=0 \quad \forall \ell_2\in \cO_2.
\end{equation}
Moreover, $t_2(\ell_2)\gtreqless \wtau_2$ if and only if $F_{01}(\ell_2)-L_{f_0}\psi(\pi \ell_2)\gtreqless 0$.
\end{lemma}

As above, we define the piecewise smooth function
\begin{equation} \label{eq: wtau2}
\tau_2(\ell_2)=\max\{ t_2(\ell_2),\wtau_2\},
\end{equation}
and the flow $\cH_3$ for $t\geq \wtau_2$ as
\begin{equation} \label{eq: flusso arco inattivo}
\cH_3(t,\ell)=
\begin{cases}
\exp((t-\wtau_2) \vec{H}_0) (\ell_2) & t\in[\wtau_2,\tau_2(\ell_2)],\\
\exp((t-\tau_2(\ell_2))\vec{F}_0)\circ
\exp((\tau_2(\ell_2)-\wtau_2) \vec{H}_0) 
(\ell_2) & t\in[\tau_2(\ell_2),\tau_2(\ell_2)+\delta(\ell_2))
\end{cases}
\end{equation}
where $\ell_2=\cK(\wtau_2,\ell)$
and $\delta(\cdot)$ is a positive function that will be specified here below. The flow $\cH_3$ enjoys the same properties of $\cH_1$, as stated in the following proposition.

\begin{proposition} \label{pro: C1 at tau2}
The flow $\cH_3$ defined above is $C^1$ and
\begin{equation}
\Phi^-(\cH_3(t,\ell))=0 \quad \forall t\in[\wtau_2(\ell),\tau_2(\ell_2)], \qquad 
\Phi^-(\cH_3(t,\ell))<0 \quad \forall t\in(\tau_2(\ell_2),\tau_2(\ell_2)+\delta(\ell_2)).
\end{equation}
Moreover, for every $t\in[\wtau_2,\wtau_2+\delta(\well_2))$ it holds
\begin{equation} \label{eq: flusso sotto 3}
\big(\pi \cH_3(t,\cdot)\big)_*|_{\well_1}=
\exp\big((t-\wtau_2)f_0 \big)_* {\big(\pi\cK_{\wtau_2}\big)_*}|_{\well_1}.
\end{equation} 
\end{proposition}
We remark that, thanks to Assumption~\ref{ass:  nonzero at switch} and by continuity, $\Phi^+(\cH_3(t,\ell))>0$ on $[\wtau_2,\wtau_2+\delta(\ell_2))$.

\paragraph{The last bang arc.}
For $t\geq\tau_2(\ell_2)$, $F_0$ is the maximised Hamiltonian until its integral curves
hit the switching surface $\{\Phi^+=0\}$. To detect the hitting time, we  solve
the implicit equation 
\begin{equation} \label{eq: def tau3}
\Phi^+\circ \exp\big((t-\tau_2(\ell_2)) \vec{F}_0\big)\circ\exp\big((\tau_2(\ell_2)-\wtau_2) \vec{H}_0\big)(\ell_2) =0,
\end{equation}
where we recall that $\ell_2=\cK(\wtau_2,\ell)\in \cO_2$.
The derivative with respect to $t$ of the left  hand side of \eqref{eq: def tau3} equals $\Rtre$
for $(t,\ell_2)=(\wtau_3,\well_2)$; thanks to Assumption~\ref{ass:  reg switch} and 
the implicit function theorem, we obtain that equation \eqref{eq: def tau3} is satisfied if and only if $(t,\ell_2)=(\tau_3(\ell_2),\ell_2)$, where $\tau_3 \colon  \cO_2 \to \R$
is a smooth function satisfying $\tau_3(\well_2)=\wtau_3$. In addition, for every $\dl \in T_{\well_2} (T^*M)$, it holds
\begin{align} \label{eq:dtau3}
\scal{\ud\tau_3(\well_2)}{\dl} &= 
 -\frac{\boldsymbol{\sigma}_{\well_3}\big( \exp\big((\wtau_3-\wtau_2)\vec{F}_0\big)_* \dl,\vPhi+\big) }{\Rtre}.
\end{align}

We choose  $\delta(\cdot) =\tau_3(\cdot)-\tau_2 (\cdot)$ in equation \eqref{eq: flusso arco inattivo}, and we consider
the following flow for $t\in[\wtau_2,T]$:
\begin{equation}
\cH_3(t,\ell)	=
\begin{cases}
\exp\big((t-\wtau_2) \vec{H}_0\big)(\ell_2) & t\in [\wtau_2,\tau_2(\ell_2)]\\
\exp\big((t-\tau_2(\ell)) \vec{F}_0\big)\circ\cH_3(\tau_2(\ell_2),\ell_2) & t \in[\tau_2(\ell_2),
\tau_3(\ell_2)]\\
\exp\big((t-\tau_3(\ell_2)) (\vec{F}_0-\vPhi+)\big)\circ
\cH_3(\tau_3(\ell_2),\ell_2) & 
t \geq
\tau_3(\ell_2).
\end{cases} 
\end{equation}
Thanks to the regularity assumptions, $\cH_3$ is an over-maximised flow for every $t\in[\wtau_2,T]$.

\medskip
Finally, the over-maximised flow $\cH  \colon  [0,T] \times \cO_1 \to T^*M$ is defined as
\begin{equation}
\cH(t,\ell)=
\begin{cases}
\cH_1(t,\ell) \quad & t\in [0,\wtau_1],\\
\cK(t,\ell) & t\in [\wtau_1,\wtau_2],\\
\cH_3(t,\ell) & t\in [\wtau_2,T].
\end{cases}
\end{equation}
Here below, we will also use the notations
\begin{equation}
\cH_t=\cH(t,\cdot),\qquad \cK_t=\cK(t,\cdot).
\end{equation}

We remark that, for every $t\in[0,T]$,  $\cH_t$ is the Hamiltonian flow associated with the Hamiltonian
\begin{equation}
H_t(\ell)=\begin{cases}
F_0(\ell)+\Phi^-(\ell) \quad & t\in[0, \tau_1(\ell)],\\
H_0(\ell)+\Phi^-(\ell) & t\in [\tau_1(\ell),\wtau_1],\\
K(\ell) & t\in [\wtau_1,\wtau_2],\\
H_0(\ell) & t\in [\wtau_2,\tau_2(\ell)],\\
F_0(\ell) & t\in[\tau_2(\ell), \tau_3(\ell)],\\
F_0(\ell)-\Phi^+(\ell) & t\in[\tau_3(\ell),T].
\end{cases}
\end{equation}

\section{Invertibility} \label{sec:invertibility}
In order to define a manifold $\Lambda_1$ such that $\pi\cH_t \colon \Lambda_1 \to M$ is locally one-to-one for every $t$, we shall also exploit the coerciveness of the extended second variation, i.e.~Assumption~\ref{ass:  coerc}. So, from now, we assume that all the Assumptions~\ref{ass:  nonzero at switch}-\ref{ass:  coerc} are satisfied.
We define $\Lambda_1$ by means of the functions  $\alpha$ and $\theta$ appearing in \eqref{eq: Jssecondo}: namely, we consider the Lagrangian submanifold
\begin{equation}
\label{Lambda 1}
\Lambda_1= \{\ud(\alpha+\theta)(q)  \colon  q \in U_{\wxu}\}.
\end{equation}
\begin{remark}
It is immediate to see that 
 $\Lambda_1\subset\Sigma^-$ and that $\vPhi-(\ell)\in T_{\ell}\Lambda_1$ 
for every $\ell\in \Lambda_1$. 
\end{remark}

This section is devoted to the proof of the following result.

\begin{proposition}
For every $t\in[0,T]$, $t \neq \wtau_3$, the flow $\pi\cH_t  \colon  \Lambda_1 \to M$ is a  local diffeormorphism from a neighbourhood of $\well_1$ onto a neighbourhood of $\wxi(t)$.   $\pi\cH_{\wtau_3}\colon  \Lambda_1 \to M$ is a locally invertible Lipschitz continuous map with Lipschitz continuous inverse.
\end{proposition}

The proof is done in several steps:  we consider separately the sub-intervals $[0,\wtau_1]$, $[\wtau_1,\wtau_2]$ and $[\wtau_2,T]$.

\medskip
\paragraph{Invertibility for $t\in[0,\wtau_1]$.}
The invertibility for $t\in[0,\wtau_1]$ is a direct consequence of equation~\eqref{eq: flusso sotto 1}.

\begin{center}
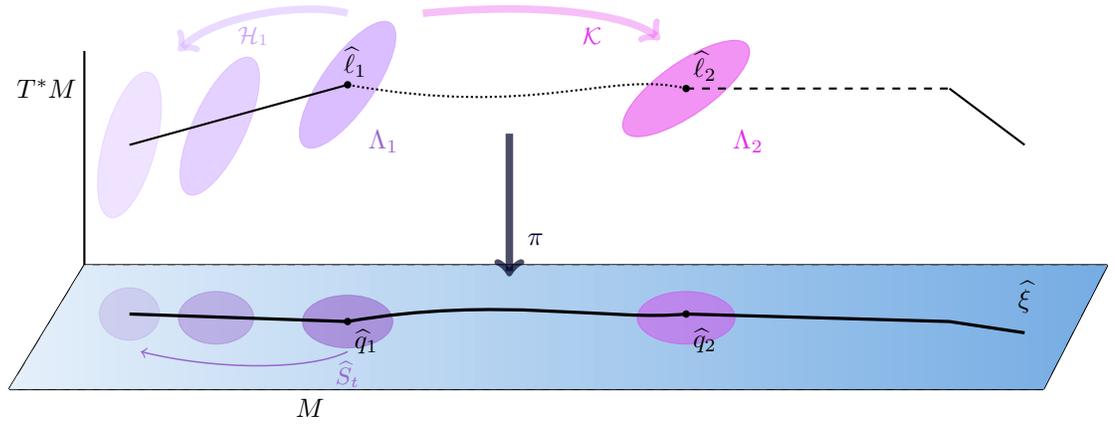
\begin{figure}[h!]
\centering
\begin{tikzpicture}[scale=0.5]
     \node[anchor=west] at (-1.7,2.6){\color{purple}$\Lambda_1$}; %

%
\draw [-,thick] (-9,-0.7)--(18.2,-0.7); %
     \draw [-,thick] (18.2,-0.7)--(16.5,-4);
     \draw [-,thick] (-11,-4)--(16.5,-4);
     \draw [-,thick] (-9,-0.7)--(-9, 5); %
     \draw [-,thick] (-9,-0.7)--(-11, -4); %
    \node[anchor=north] at (-10,4.5){$T^*M$}; %
    \node[anchor=north] at (-3,-4){$M$};

          \shade[left color=turq!10,right color=turq!70] 
    (-9,-0.7) to[out=0,in=0] (18.2,-0.7) -- (16.5,-4) to[out=0,in=0]
    (-11,-4) -- cycle;

\draw [fill=lavanda,lavanda,opacity=0.7,rotate around={55:(-2,4.1)}] (-2,4.1)  ellipse (2cm and 0.7cm);

            
\draw [fill=purple,purple,opacity=0.6] (-2,-2.2)  ellipse (1.2cm and 0.7cm);

\draw[->,color=dblue,line width=1mm,opacity=0.7] (2.3,2.8) -- (2.3,-1);
\node  at (3,0) {\color{dblue}$\pi$} ;
%

\draw [fill=lavanda,lavanda,opacity=0.5,rotate around={65:(-5.4,3)}] (-5.4,3)  ellipse (2cm and 0.7cm);             

\draw [fill=lavanda,lavanda,opacity=0.3,rotate around={75:(-7.8,2.5)}] (-7.8,2.5)  ellipse (2cm and 0.7cm);

            
\draw [fill=purple,purple,opacity=0.4] (-5.5,-2.1)  ellipse (1cm and 0.7cm); 
            
\draw [fill=purple,purple,opacity=0.2] (-7.8,-2)  ellipse (0.8cm and 0.7cm);

\draw [fill=darkpink,darkpink,opacity=0.5,rotate around={35:(7,4)}] (7,4)  ellipse (2cm and 0.7cm);
            
    \node[anchor=west] at (8,2.6){\color{darkpink}$\Lambda_2$}; %
            
            
\draw [fill=darkpink,darkpink,opacity=0.4] (7,-2.1)  ellipse (1.3cm and 0.7cm);


\draw[thick](-7.8,2.5)--(-2,4.1); 
\draw[thick,densely dotted](7, 4) .. controls (5, 4.5) and (3,3.2) ..
     (-2,4.1);    
\draw[thick, dashed] (7,4)--(14,4) ;
\draw[thick](14,4) -- (16,2.5); 

\draw[very thick](-7.8,-2)--(-2,-2.2);
\draw [very thick](7, -2) .. controls (5, -2.2) and (2,-1.5) ..     (-2,-2.2); 
\draw[very thick] (7,-2)--(14,-2.2) -- (16,-2.5); 
\node [anchor=south] at (16,-2.2){$\wxi$};

         \draw [line width=1mm,opacity=0.4,color=lavanda,->] (-2,6) .. controls (-3,6.3) and (-5.5,6).. (-6.5,5);
  \node  at (-4.5,5.4) {\color{lavanda}\small
  $\mathcal{H}_1$} ;

         \draw [line width=1mm,opacity=0.3,color=darkpink,->] (0,6) .. controls (3,6.3) and (5.5,6).. (6.3,5.3);
  \node  at (4.5,5.4) {\color{darkpink}\small
  $\mathcal{K}$} ;

         \draw [line width=0.2mm,color=purple,->] (-2,-3) .. controls (-3,-3.5) and (-5.5,-3.5).. (-7.5,-3);
  \node  at (-2,-3.6) {\color{purple}\small
  $\wh{S}_t$} ;

  \draw [thick,color=black,fill=black] (-2,4.1) circle (0.07cm);
    \node at (-1.8,4.7) {$\well_1$};

    \draw [thick,color=black,fill=black] (-2,-2.2) circle (0.07cm) ;
\node [anchor=south] at (-1.5,-3.3){$\wq_1$}; %

     \draw [thick,color=black,fill=black] (7,4) circle (0.07cm);
    \node at (7.5,4.6) {$\well_2$};   

     \draw [thick,color=black,fill=black] (7,-2) circle (0.07cm);
    \node at (7.5,-2.7) {$\wq_2$};

  \end{tikzpicture}
  \caption{Invertibility for $t\in[0,\wtau_3)$: $\mathcal{H}_t(\Lambda_1)$ projects diffeormorphically onto a neighbourhood of $\wxi(t)$ for every $t\in[0,\wtau_3)$.\\  
  Solid lines denote the bang arcs, dashed lines inactivated arcs, and dotted lines singular arcs.}
  \end{figure}
\end{center}

\paragraph{Invertibility for $t\in[\wtau_1,\wtau_2]$.}
In order to prove the claim, we introduce the auxiliary Hamiltonian
\begin{equation} 
 \wh H_t=H_0+\wu(t) \Phi^-, \qquad t\in[\wtau_1,\wtau_2]. 
\end{equation}
This new Hamiltonian shares some important features with $K$. In particular, $\wh H_t$ is an over-maximised Hamiltonian on $\Sigma^-$ too, and its Hamiltonian vector field is tangent to $\Sigma^-$;
we denote with $\wh{\cH}_t$
its flow from time $\wtau_1$ to time $t$. The invertibility of $\pi\wh{\cH}_t|_{\Lambda_1}$ is related to the invertibility of
$\pi{\cK}_t|_{\Lambda_1}$, as the following result shows.

The proof uses the same arguments of \cite[Lemma~9]{PS11}; we sketch it in the Appendix.
\begin{lemma} \label{lemma:  K e H uguale}
 For every $t\in[\wtau_1,\wtau_2]$, the followings hold
\begin{enumerate}
	\item $\wh{\cH}_{t*}(T_{\well_1}\Lambda_1)=\cK_{t*}(T_{\well_1}\Lambda_1)$.
	\item If  $\ker\big(\pi\wh{\cH}_t\big)_*|_{T_{\well_1}\Lambda_1}=0$,  then $\ker\big(\pi{\cK}_t\big)_*|_{T_{\well_1}\Lambda_1}=0$.
\end{enumerate} 
\end{lemma}

Taking advantage of Lemma~\ref{lemma:  K e H uguale}, 
the invertibility of $\pi\wh{\cH}_t|_{\Lambda_1}$ implies the one of
$\pi\wh{\cK}_t|_{\Lambda_1}$. On the other hand, it turns out that $\wh{\cH}_t$ is directly linked to the Hamiltonian flow associated with the second variation (see details here below); therefore, it is much easier to prove the invertibility of $\pi\wh{\cH}_t|_{\Lambda_1}$ as a consequence of the coerciveness of the second variation. Indeed, consider the subspace
$\bcV\subset\bcW$ defined by
\begin{equation}
\bcV  := \big\{
	\de \in \bcW  \colon  \ep = 0
	\big\}. \end{equation}	
\begin{lemma} \label{lemma:  H inv}
Assume  that $\bJ|_{\bcV}$ is coercive. Then $\ker\big(\pi\wh{\cH}_t\big)_*|_{T_{\well_1}\Lambda_1}=0$ for any $t\in[\wtau_1,\wtau_2]$.
\end{lemma}

\begin{proof}
We consider the LQ optimal control problem on $T_{\wxu}M$ given by
\begin{equation}\label{eq:LQ}
\min_{\de\in \bcV}  \bJ[\de]^2.
\end{equation}
The maximised Hamiltonian associated by PMP with this LQ problem is
\begin{equation}
H_t''(\delta p,\delta z)= \frac{1}{2 R(t)} \Big(\langle \delta p,\dot{g}_t^1(\wxu)\rangle-L_{\delta z} A(t,\wxu)  \Big)^2 \qquad (\delta p,\delta z)\in T_{\wxu}^*M\times T_{\wxu}M.
\end{equation}
Since $\de \in \bcV$, $\delta z$ is free so that PMP  applied to problem \eqref{eq:LQ} gives the following  transversality conditions at the initial point
\begin{equation}
(\delta p,\delta z)\in L_{\wtau_1}'':=\{(\delta p,\delta z) \colon  \delta z\in T_{\wxu}M, \delta p=-\uD^2\big(\theta + \alpha + \wh\gamma_T + \wS{T}^0 \big)[\delta z,\cdot]\}.
\end{equation}
Denote with $\cH_t''$  the Hamiltonian flow of $H_t''$.
In order to compare $\cH_t\se$  with $\widehat{\cH}_t$, we define
the anti-symplectic isomorphism
$\iota  \colon  T_{\wxu}^*M\times T_{\wxu}M \to T_{\well_1}(T^*M)$ as
\begin{equation}\label{eq: anti sym iso}
\iota(\delta p,\delta x)=-\delta p+\ud\big( - \wh{S}_T^0-\wh{\gamma}_T \big)_* \delta x. 
\end{equation}
By definition, $\boldsymbol{\sigma}\circ \iota\otimes \iota=-\wh{\sigma}$, where $\wh{\sigma}$ denotes the standard symplectic structure on $T_{\wxu}^*M\times T_{\wxu}M$.
It is immediate to verify that 
\begin{equation}
\iota L_{\wtau_1}'' 
=T_{\well_1}\Lambda_1.
\end{equation}
Moreover, by analogous computations to those in \cite{CS10,PS11}, it is easy to prove that
\begin{equation} \label{eq: hami hami}
\cH_{t}\se=\iota^{-1} \widehat{\cF}_{t *}^{-1} \wh{\cH}_{t *}\iota \quad t\in[\wtau_1,\wtau_2], 
\end{equation}
which implies that
\begin{equation}\label{eq:isoinversa}
\left(\pi\cH_t\se\right)^{-1} = 
\iota^{-1}\big(\pi\wh\cH_t\big)^{-1}_* \wS{t*}.
\end{equation}

On the other hand, \cite[Theorem~2.6]{SZ97} states that $\bJ|_{\bcV}$ is coercive if and only if $\pi \cH_t''  \colon  L_{\wtau_1}''\to T_{\wxu}M$ is one to one for 
every $t\in [\wtau_1,\wtau_2]$,
so that the coerciveness of $\bJ|_{\bcV}$ implies that $\big(\pi \wh{\cH}_{t}\big)_*|_{T_{\well_1}\Lambda_1}$ is invertible for every $t\in[\wtau_1,\wtau_2]$.
\end{proof}

Coupling Lemma~\ref{lemma:  K e H uguale} with Lemma~\ref{lemma:  H inv}, we obtain that, if $\bJ|_{\bcV}$ is coercive, then 
$(\pi\cK_t)|_{\Lambda_1}$ is invertible for $t\in[\wtau_1,\wtau_2]$.

\begin{remark} \label{Lambda 2}
Set $\Lambda_2:=\cK_{\wtau_2}(\Lambda_1)$.
Since $\Lambda_2$ is a Lagrangian submanifold of 
$T^*M$ which projects one to one onto a neighbourhood of $M$, there exist a neighbourhood $U_{\wq_2}$ of $\wq_2$ and a smooth function $\alpha_2  \colon  U_{\wq_2} \to \R$ such that
\begin{equation}
\ud\alpha_2(\wx_2)=\well_2, \qquad
\Lambda_2:={\cK}_{\wtau_2} (\Lambda_1)=\{\ud\alpha_2(q) \colon  q \in U_{\wq_2}\}.
\end{equation}
\end{remark}

\smallskip

\begin{center}
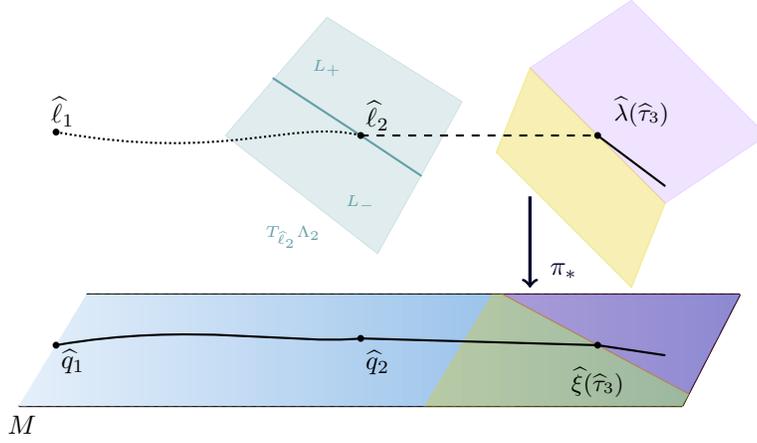
\begin{figure}[h!]
\begin{tikzpicture}[scale=0.5]
\pgftransformscale{0.9} %

%
\draw [-,thick] (-1.1,-0.7)--(18.2,-0.7); %
     \draw [-,thick] (18.2,-0.7)--(16.5,-4);
     \draw [-,thick] (-3.1,-4)--(16.5,-4);
    \node[anchor=north] at (-3,-4){$M$};

          \shade[left color=turq!10,right color=turq!70] 
    (-1.1,-0.7) to[out=0,in=0] (18.2,-0.7) -- (16.5,-4) to[out=0,in=0]
    (-3.1,-4) -- cycle;


\draw[->,color=dblue,very thick] (12,2.2) -- (12,-0.5);
\node  at (13,0) {\color{dblue}$\pi_*$} ;


\draw[fill=orange,princ,opacity=0.2]  (3,4)--(6,7.5)--(10,5)--(7.5,0.5)--(3,4) -- cycle;

\draw[princ,thick] (4.4,5.7)--(7,4)--(8.8,2.8);
\node at (5,1){\tiny\color{princ}$T_{\well_2}\Lambda_2$};
\node at (6,6){\tiny\color{princ}$L_+$};
\node at (7,2){\tiny\color{princ}$L_-$};

\draw[fill=lavanda,lavanda,opacity=0.3]  (12,6)--(16,2) -- (19,4)--(15,8)--(12,6) -- cycle;

\draw[fill=dor,dor,opacity=0.3]  (12,6)--(16,2) -- (15,-0.5)--(11,3.5)--(12,6) -- cycle;

            
 \draw[very thin,color=gray] (11.2,-0.7)--(16.7,-3.65);           
%

\draw[fill=magenta,magenta,opacity=0.2]  (11.2,-0.7)--(16.7,-3.65) -- (18.2,-0.7)--(12.9,-0.7) -- cycle;
\draw[fill=dor,dor,opacity=0.3]  (11.2,-0.7)--(16.7,-3.65) -- (16.5,-4)--(8.9,-4)--(10.9,-0.7) -- cycle;

\draw[thick,densely dotted](7, 4) .. controls (5, 4.5) and (3,3.2) ..
     (-2,4.1);    
\draw[thick,dashed] (7,4)--(14,4);
\draw[thick] (14,4)-- (16,2.5); 

\draw [thick](7, -2) .. controls (5, -2.2) and (2,-1.5) ..     (-2,-2.2); 
\draw[thick] (7,-2)--(14,-2.2) -- (16,-2.5);

  \draw [thick,color=black,fill=black] (-2,4.1) circle (0.07cm);
    \node at (-1.8,4.7) {$\well_1$};

    \draw [thick,color=black,fill=black] (-2,-2.2) circle (0.07cm) ;
\node [anchor=south] at (-1.5,-3.3){$\wq_1$}; %

     \draw [thick,color=black,fill=black] (7,4) circle (0.07cm);
    \node at (7.5,4.6) {$\well_2$};   

     \draw [thick,color=black,fill=black] (7,-2) circle (0.07cm);
    \node at (7.5,-2.7) {$\wq_2$};

    \draw [thick,color=black,fill=black] (14,4) circle (0.07cm) ;
\node [anchor=south] at (15.3,4){\small$\wla(\wtau_3)$}; %

    \draw [thick,color=black,fill=black] (14,-2.2) circle (0.07cm) ;
\node [anchor=south] at (14,-4){\small$\wxi(\wtau_3)$}; %

  \end{tikzpicture}
  \caption{Invertibility at $t=\wtau_3$.
$L_+$ (respectively, $L_-$) denotes the half space of all $\delta \ell_2\in T_{\well_2}\Lambda_2$ such that $\langle \ud \wtau_3(\well_2),\delta \ell_2\rangle \geq 0$ (respectively, $\leq 0$).\\
In purple, $\mathcal{H}_{\wtau_3 *}L_+$; in yellow,  $\mathcal{H}_{\wtau_3 *}L_-$. These semi-planes project without intersections on $T_{\wxi(\wtau_3)}M$.
  }
  \end{figure}
\end{center}

\smallskip

\paragraph{Invertibility for $t\in[\wtau_2,T]$.}
Thanks to equation~\eqref{eq: flusso sotto 3}, we obtain that,  if $\bJ|_{\bcV}$ is coercive, then $\pi\cH_t|_{\Lambda_1}$ is locally  invertible for every $t< \wtau_3$.

\smallskip
Let $\dl\in T_{\well_1}\Lambda_1$, and set $\dl_2=\cK_{\wtau_2*}\dl$.
The first order approximation of $\pi\cH_{\wtau_3}$ at $\well_1$, applied to $\dl$, is given by
\begin{equation}
\begin{cases}
\exp\big((\wtau_3-\wtau_2) f_0 \big )_* \pi_* \dl_2 & \mbox{ if } 
\scal{\ud\tau_3(\well_2)}{\dl_2} \geq 0,\\
\scal{\ud\tau_3(\well_2)}{\dl_2} f_1(\wx_3)+ \exp\big((\wtau_3-\wtau_2) f_0 \big )_* \pi_* \dl_2 & \mbox{ if } 
\scal{\ud\tau_3(\well_2)}{\dl_2} \leq 0,
\end{cases}
\end{equation}
which, up to a pullback, can be written as 
\begin{equation}
\begin{cases}
\pi_* \dl_2 & \mbox{ if } \scal{\ud\tau_3(\well_2)}{\dl_2} \geq 0,\\
\pi_* \dl_2 -  \scal{\ud\tau_3(\well_2)}{\dl_2}  \wt k(\wx_2)
& \mbox{ if } 
\scal{\ud\tau_3(\well_2)}{\dl_2} \leq 0,
\end{cases}
\end{equation}
where $ \wt k := \wh{S}_{\wtau_2*}k$ is the pullback of $-f_1$ from  time $\wtau_3$ to time $\wtau_2$. Notice that, by Assumption~\ref{ass:  nonzero at switch}, $f_1(\wq_3) \neq 0$, so that $\wt k(\wq_2) \neq 0$.

By Clarke's inverse function theorem \cite{Cla76}, $\pi\cH_{\wtau_3}$ is invertible if for every $a\in[0,1]$ and 
for every $\dl_2 \in T_{\well_2}\Lambda_2$ 
it holds
\begin{equation} \label{eq: clarke contrad}
\pi_* \dl_2 - a  \scal{\ud\tau_3(\well_2)}{\dl_2} \wt k(\wx_2) \neq 0. 
\end{equation}
By contradiction, assume there exist some $a\in[0,1]$ and $\dl_2 \in 
T_{\well_2}\Lambda_2$, $\dl_2\neq 0$, such that the left hand side of \eqref{eq: clarke contrad} is zero. 
This implies  that $\pi_*\dl_2=\rho \wt k(\wx_2)$, for some $\rho\neq 0$, so that 
\begin{equation}
\rho  \wt k(\wx_2) -a  \scal{\ud\tau_3(\well_2)}{\dl_2}  \wt k(\wx_2)=0,
\end{equation}
which yields 
\begin{equation}
1-a \scal{\ud\tau_3(\well_2)}{\ud\alpha_{2*} \wt k(\wx_2)} = 0,
\end{equation}
since $\dl_2=\rho \ud\alpha_{2*} \wt k(\wx_2)$. Equivalently, using \eqref{eq:dtau3}, 
\begin{equation} \label{eq: clarke}
\Rtre+ a \sigma_{\well_3} \big(\exp((\wtau_3-\wtau_2)\vec{F}_0)_* 
\ud\alpha_{2*} \wt k(\wx_2),\vPhi+\big)=0.
\end{equation}
By Assumption~\ref{ass:  reg switch}, the left hand side of \eqref{eq: clarke} is negative for $a=0$. If we show that it is negative also for $a=1$, then, by linearity, we get a contradiction and we are done. The last part of this section is devoted to prove that this is a consequence of the coerciveness of the second variation
on $\bcW$.

\smallskip
We denote with the symbol $\boldsymbol{J^b}$ the bilinear form associated with $\bJ$.
We recall that $\bJ$ is coercive on $\bcW$ if and only if it is coercive both on $\bcV$ and on $\bcW\cap \bcV^{\bot}$,
where $ \bcV^{\bot}$ denotes the orthogonal complement to $ \bcV$ with respect to 
$\boldsymbol{J^b}$.
In order to compute $\bcV^{\bot}$, we introduce, for every $\de=(\delta z,\ep,w) \in \bcW$, the trajectory $p \colon [\wtau_1,\wtau_2] \to T_{\wxu}^*M$  solution of
the Cauchy problem
\begin{equation}
 \begin{cases}
 \dot{p}(t)=-w(t)L_{\cdot} A(t,\wxu), \\
 p(\wtau_1)=-\uD^2(\alpha+\theta+\wh\gamma_T+\wh S_T^0)(\wxu) [\delta z,\cdot].
 \end{cases}
\end{equation}
Let $\wt\de=(\wt{\delta z},\wt\ep,\wt w)\in \bcW$ be another admissible variation; then $\boldsymbol{J^b}[\de,\wt\de]$ can be written as 
\begin{equation} \label{eq:bilinear}
\begin{split} 
\boldsymbol{J^b}[\de,\wt\de] & =  \dfrac{1}{2}\uD^2\big(\theta + \alpha + \wh\gamma_T + \wS{T}^0 \big) (\wxu)[\dz,\wt\dz]\\& + 
\dfrac{1}{2}\int_{\wtau_1}^{\wtau_2} \big( w^2(t)\wt w(t) R(t) + w(t) L_{\wt\zeta(t)} A(t,\wxu) + \wt w(t) L_{\zeta(t)} A(t,\wxu)\big) \ud t \\
&+\dfrac{\ep \wt\ep}{2} \Big(L_{[k_4,k_3]} \big(\wh\gamma_T + \wS{T}^0 - \wS{\wtau_2}^0\big) -L_{k_4} \wpsi{\wtau_3} (\wxu)   -\liederdo{k}{ \big(\wh\gamma_T + \wS{T}^0 - \wS{\wtau_2}^0\big)}  
\Big) \\ 
& =  \dfrac{1}{2}\langle \uD^2\big(\theta + \alpha + \wh\gamma_T + \wS{T}^0 \big) (\wxu)[\delta z,\cdot] + 
p(\wtau_1) , \wt{\delta z} \rangle-\dfrac{1}{2} \langle p(\wtau_2),\wt\zeta(\wtau_2)\rangle\\
&+\dfrac{1}{2}\int_{\wtau_1}^{\wtau_2} \wt w(t) \Big( w(t) R(t) + L_{\zeta(t)} A(t,\wxu) + \langle p(t),\dot{g}_t^1 \rangle \Big) \ud t \\
&+\dfrac{\ep \wt\ep}{2} \Big(L_{[k_4,k_3]} \big(\wh\gamma_T + \wS{T}^0 - \wS{\wtau_2}^0\big) -L_{k_4} \wpsi{\wtau_3} (\wxu)   -\liederdo{k}{ \big(\wh\gamma_T + \wS{T}^0 - \wS{\wtau_2}^0\big)}  
\Big) .
\end{split}
\end{equation}
From the expression here above, we see that $\delta e\in \bcV^{\bot}$ if and only if 
\begin{equation} \label{eq: ctr feed lin}
 w(t) R(t) + L_{\zeta(t)} A(t,\wxu) + \langle p(t),\dot{g}_t^1 \rangle =0 \qquad \forall t\in[\wtau_1,\wtau_2],
\end{equation}
so that, for any $\delta e\in \bcV^{\bot}$, we get that
\begin{equation}
\bJ[\de]^2  =  -\dfrac{1}{2} \langle p(\wtau_2),\zeta(\wtau_2)\rangle
+\dfrac{\ep^2}{2} \Big(L_{[k_4,k_3]} \big(\wh\gamma_T + \wS{T}^0 - \wS{\wtau_2}^0\big) -L_{k_4} \wpsi{\wtau_3} (\wxu)   -\liederdo{k}{ \big(\wh\gamma_T + \wS{T}^0 - \wS{\wtau_2}^0\big)}  
\Big) .
\end{equation}

\begin{remark} \label{rem: Hse feedback}
We stress that the solution $(p(t),\zeta(t))$ of the Cauchy problem
\begin{equation}
\begin{cases}
\dot{\zeta}(t)=w(t)\dot{g}_t^1(\wxu),\\
\dot{p}(t)=-w(t)L_{(\cdot)}A(t,\wxu),\\
(p(\wtau_1),\zeta(\wtau_1))\in L_{\wtau_1}\se,
\end{cases}
\end{equation}
associated with the control
$w(\cdot)$ that satisfies equation~\eqref{eq: ctr feed lin}
is also the solution of the Hamiltonian system associated with $H_t\se$ with the same initial condition, that is
\begin{equation}
(p(t),\zeta(t))=\cH_t\se (p(\wtau_1),\zeta(\wtau_1)).
\end{equation}
\end{remark}

Remark that $\bcV^{\bot}$ is a 1-dimensional linear space: indeed, for every $\ep$, $\zeta(\tau_2)=-\ep k(\wxu)$, so that 
$\delta z$ is uniquely determined as the backward solution of the linear system with control \eqref{eq: ctr feed lin}.
By homogeneity, we can choose $\delta e\in \bcV^{\bot}$ with $\ep=-1$. Then $\bJ|_{\bcV^{\bot}}$ is coercive if and only if
\begin{align} \label{eq: zero maggiore}
0&> \scal{p(\wtau_2)}{\zeta(\wtau_2)}
-L_{[k_4,k_3]} \big(\wh\gamma_T + \wS{T}^0 - \wS{\wtau_2}^0\big)(\wxu) +L_{k_4} \wpsi{\wtau_3}(\wxu)   +
\liederdo{k}{ \big(\wh\gamma_T + \wS{T}^0 - \wS{\wtau_2}^0\big)}    \nonumber \\
&=\scal{p(\wtau_2)}{k(\wxu)} + \Rtre +L_k \abs{\wpsi{\wtau_3}}(\wxu) +L_k^2 
(\wh \gamma_T+\wh S^0_T -\wh S^0_{\wtau_2} )(\wxu), 
\end{align}
where the expression~\eqref{eq: zero maggiore} is obtained applying \eqref{eq: solution lambda} and the definition of $\wh\gamma_T$.

We now compute $\scal{p(\wtau_2)}{k(\wxu)}$ in terms of Hamiltonian flows. 
Consider the pair $(\delta p,\delta z)\in L_{\wtau_1}\se$
such that $\pi\cH_{\wtau_2}\se(\delta p,\delta z)=k(\wxu)$
(thanks to the invertibility of $\pi\cH\se_{\wtau_2}$, it exists and it is unique), so that, by \eqref{eq:isoinversa}, we get 
$(\delta p,\delta z)= \iota^{-1} (\pi \wh{\cH}_{\wtau_2})_*^{-1}  \widehat{S}_{\wtau_2 *} k(\wxu)$. 
Thanks to \eqref{eq: hami hami} we obtain\begin{align} 
\scal{p(\wtau_2)}{k(\wxu)} & = \wh\sigma\big(\cH_{\wtau_2}\se (\pi\cH_{\wtau_2}\se)^{-1}k(\wxu),(0,k(\wxu)) \big)\\
&= - \boldsymbol{\sigma}_{\well_1} \Big(
\widehat{\cF}_{\wtau_2*}^{-1} \widehat{\cH}_{\wtau_2*} \iota (\pi\cH_{\wtau_2}\se)^{-1}k(\wxu), 
 \ud \big(-\wh \gamma_T-\wh S_T^0\big)_* k(\wxu)
 \Big)\\
&= - \boldsymbol{\sigma}_{\well_1} \Big( \widehat{\cF}_{\wtau_2*}^{-1} 
\ud\alpha_{2*} 
\wt k(\wx_2) , 
\ud \big(-\wh \gamma_T-\wh S_T^0\big)_* k(\wxu)
\Big),\label{eq:ptau2k}
\end{align}
where  $\alpha_2$ is the function defined in Remark~\ref{Lambda 2}.
Indeed, thanks to \eqref{eq:isoinversa}, we obtain that $\iota (\pi\cH_{\wtau_2}\se)^{-1}k(\wxu)=
(\pi \wh{\cH}_{\wtau_2})^{-1}_* \wt k(\wx_2)$, so that  $\iota (\pi\cH_{\wtau_2}\se)^{-1}k(\wxu)\in T_{\well_1}\Lambda_1$ and
\begin{equation}
\wh{\cH}_{\wtau_2*}\iota (\pi\cH_{\wtau_2}\se)^{-1}k(\wxu)\in \wh{\cH}_{\wtau_2*}(T_{\well_1}\Lambda_1)=
\cK_{\wtau_2*}(T_{\well_1}\Lambda_1)=T_{\well_2}\Lambda_2.
\end{equation}
By cumbersome but standard computations,  
$\widehat{\cF}_{\wtau_2*} \ud \big(-\wh \gamma_T-\wh S_T^0\big)_* k(\wxu)
= \ud \big((-\wh \gamma_T-\wh S_T^0+\wh S_{\wtau_2}^0)\circ \wh S_{\tau_2}^{-1}\big)_* 
\wt k(\wx_2)$. 
%
Equation \eqref{eq:ptau2k} thus  gives
\begin{align*}
\scal{p(\wtau_2)}{k(\wxu)} & =
- \boldsymbol{\sigma}_{\well_2} \big( \ud\alpha_{2*}\wt k(\wx_2), 
\ud \big((-\wh \gamma_T-\wh S_T^0+\wh S_{\wtau_2}^0)\circ \wh S_{\tau_2}^{-1}\big)_* 
\wt k(\wx_2)   \big)\\
&= -\uD^2 \big(\alpha_2+(\wh \gamma_T+\wh S^0_T -\wh S^0_{\wtau_2} ) \circ \wh S_{\wtau_2}^{-1}\big)[\wt k(\wx_2)]^2.
\end{align*}
Finally,  computing the value of $\well_2$ by means of \eqref{eq: solution lambda}, we get
\begin{align}
\langle p(\wtau_2),& \, k(\wxu)\rangle  + L^2_k(\wh \gamma_T
+\wh S^0_T -\wh S^0_{\wtau_2} )(\wxu) \\
&= -\uD^2 \big(\alpha_2+(\wh \gamma_T+\wh S^0_T -\wh S^0_{\wtau_2} )
\circ \wh S_{\wtau_2}^{-1}\big)[\wt k(\wx_2)]^2 + L_{\wt k}^2
\big( (\wh \gamma_T+\wh S^0_T -\wh S^0_{\wtau_2} )\circ \wh S_{\wtau_2}^{-1})(\wx_2)\\
&= -\uD^2 \alpha_2 [\wt k(\wx_2)]^2 - \well_2 \uD\wt k (\wx_2)\wt k(\wx_2) 
= -L^2_{\wt k}\alpha_2(\wx_2).
\end{align}
Plugging this equality in equation~\eqref{eq: zero maggiore}, we obtain that
\begin{equation} \label{eq: coerc}
 \Rtre +L_k \abs{\wpsi{\wtau_3}}(\wxu) - L^2_{\wt k}\alpha_2(\wx_2) <0.
\end{equation}
It now suffices to notice that 
\begin{align}
\sigma_{\well_3} \big(\exp((\wtau_3-\wtau_2)\vec{F}_0)_*  \ud\alpha_{2*}\wt k(\wx_2),\vPhi+\big)&=
\sigma_{\well_2} \Big(d\alpha_{2*}\wt k(\wx_2),\exp(-(\wtau_3-\wtau_2)\vec{F}_0)_* \vPhi+)\Big)\\
&= -L_{\wt k}^2 \alpha_2 (\wx_2) +L_k \wpsi{\wtau_3}(\wxu),
\end{align}
to obtain that 
equation~\eqref{eq: clarke} holds true for $a=1$, and we are done.

\section{Main results} \label{sec:  main results}

We now state the main result of the paper. This section is devoted to its proof. 
\begin{theorem} \label{th: main result}
Let $\wxi$ be an admissible trajectory for system \eqref{eq: contr sys} and satisfying Assumptions~\ref{ass:  nonzero at switch}--\ref{ass:  coerc}. 
	Then $\wxi$ is a strict strong local minimiser of problem \eqref{eq: cost}-\eqref{eq: contr sys}. 
\end{theorem}

The proof of the strong local optimality of the reference trajectory
 follows the same lines of \cite[Theorem~4.1]{CP19} (see also \cite{SZ16}), 
thus we are just recalling the main arguments. We shall instead provide all details for the proof of the strictness part.
\begin{proof}
We define on $\R\times T^*M$  the one-form  
$\omega=\cH^{*}_t \varsigma - H_t \circ \cH_t \ud t$. Applying \cite[Lemma~3.3]{SZ16} we can prove that $\omega$ is exact on
$[0,T]\times \Lambda_1$.

Let $\eulO\subset \mathbb{R}\times M$ be a neighbourhood of the graph of $\wxi$ such that 
$\pi\cH \colon [0,T]\times \Lambda_1\to \eulO$ is invertible, with piecewise $C^1$ inverse.
Consider an admissible trajectory 
$\xi  \colon  [0,T]\to M$ of \eqref{eq: contr sys} whose graph is contained in $\eulO$, and call $v(t)$  its associated control;  set
\begin{equation}
\ell(t) =(\pi\cH_t)^{-1}(\xi(t)), \qquad 
\lambda(t)=\cH_t(\ell(t)).
\end{equation}
We define the closed path $\gamma  \colon [0,2T]\to [0,T]\times \Lambda_1$ as
\begin{equation}
\gamma(t)=
\begin{cases}
(t,\ell(t)) & t\in[0,T],\\
(2T-t,\well_1) & t\in[T,2T].
\end{cases}
\end{equation}
%
Integrating $\omega$ along $\gamma$, and recalling that $H_t$ is an over-maximised Hamiltonian, we obtain that
\begin{equation} \label{eq: costo minore uguale}
\int_0^T |\wu(t)\psi(\wxi(t))|\ud t\leq \int_0^T |v(t)\psi(\xi(t))|\ud t,
\end{equation}
that is, $\wxi$ is a strong local minimiser.

Assume now that $\wxi$ is not a strict minimiser, that is, there exists an admissible trajectory $\xi$ with graph in $\eulO$ for which equality holds in equation~\eqref{eq: costo minore uguale}; this is equivalent to
\begin{equation} \label{eq: stretto Ham}
\langle \lambda(t),(f_0+v(t) f_1)(\xi(t))\rangle -|v(t)\psi(\xi(t))|=H_t(\lambda(t)) \qquad \mbox{ a.e. } t\in[0,T], 
\end{equation}
that is, both $\langle \lambda(t),(f_0+v(t) f_1)(\xi(t))\rangle -|v(t)\psi(\xi(t))|$ and $H_t(\lambda(t))$ coincide with $H_{\max}(\lambda(t))$. Since $\xi(0)=\wxi(0)$, and by regularity of the first bang arc, then, for $t$ small enough, $h(\lambda(t),u)$ attains its maximum only for $u=\wu(t)=1$. This implies that $v(t)=\wu(t)$ and $\xi(t)=\wxi(t)$ as long as 
$\Phi^-(\lambda(t))>0$, that is, for $t\in[0,\wtau_1)$.

Analogously, since $\xi(T)=\wxi(T)$, we can apply, backward in time, an analogous argument, and prove that $u(t)\equiv \wu(t)$ for $t\in(\wtau_2,T]$ (see \cite{CP19} for  more details).

\medskip
For $t\in[\wtau_1,\wtau_2]$, equation \eqref{eq: stretto Ham} implies that $\lambda(t)\in S^-$. Thus, since $\vec{K}$ is tangent to $S^-$ (see 
Proposition~\ref{pro: K tangente}), then  
$\ell(t)\in S^-$ too. 
We claim that there exists a function $a \colon [\wtau_1,\wtau_2] \to \R$ such that
\begin{equation} \label{eq: dot ell}
\dot{\ell}(t)=-a(t)\vPhi-(\well_1).
\end{equation}
If so, since $\dot{\ell}(t)$ is tangent to $S^-$ and $\vPhi-$ is transverse to $S^-$, we obtain that $a(t)\equiv 0$ and $\dot{\ell}(t)\equiv 0$
for $t\in[\wtau_1,\wtau_2]$, that is $\xi(t)=\pi \cK_t( \well_1)=\wxi(t)$, which completes the proof.

In order to prove \eqref{eq: dot ell}, we first observe that
\begin{gather}  \label{derivata xi}
\dot{\xi}(t)=\pi_*\dot{\lambda}(t)=\pi_*\big(\vec{K} (\lambda(t)) + \cK_{t*} \dot{\ell}(t)\big) .
\end{gather}
On the other hand, it is not difficult to see that, for every $t \in [\wtau_1, \wtau_2]$, the function
\begin{equation}
\Sigma^-\cap T_{\xi(t)}^*M \ni \ell
\mapsto K(\ell) - \langle \ell, f_0(\xi(t))+v(t) f_1(\xi(t))\rangle + |v(t) \psi(\xi(t))| 
\end{equation}
has a minimum at $\ell=\lambda(t)$; thus, its derivative along variations in $\Sigma^-\cap T_{\xi(t)}^*M$ (that is, $\delta p$ with
$\langle \delta p, f_1(\xi(t))\rangle=0$) must be zero. This means that the derivative 
with respect to the vertical coordinates (i.e., the directions contained in $T_{\xi(t)}^*M $) must be parallel to $f_1(\xi(t))$, which implies that
\begin{equation} \label{derivata dep}
\pi_* \vec{K} (\lambda(t))=f_0(\xi(t))+v(t)f_1(\xi(t))+a(t)f_1(\xi(t))=\dot{\xi}(t)+a(t)f_1(\xi(t)),
\end{equation}
for some real function $a(\cdot)$.
Combining \eqref{derivata xi} with \eqref{derivata dep}, we obtain that
\begin{equation}
(\pi\cK_t)_* \dot{\ell}(t)+a(t)f_1(\xi(t))=0.
\end{equation}
Since $(\pi \cK_{t})_*^{-1}f_1(\xi(t)) = \vPhi{-}(\well_1)$, applying 
$(\pi \cK_{t})_*^{-1}$ to the equality here above we get the claim.
\end{proof}

\section{Example}\label{sec:example}

In this section, we apply our result to the optimal control problem in $\R^2$
\begin{equation}
\min_{|u(\cdot)|\leq 1} \int_0^T |u(t)x_2(t)|\;\ud t
\end{equation}
subject to the control system
\begin{equation}
\begin{cases}
\dot{x}_1=x_2,\\
\dot{x}_2=u-\attr x_2,\\
x_1(0)=0,\ x_2(0)=0,\\
x_1(T)= X>0, \ x_2(T)=0,
\end{cases}
\end{equation}
where $T$, $\attr$ and $X$ are given positive constants.
This problem, studied in \cite{boizot}, models the fuel consumption minimisation problem for an academic electric vehicle moving in one horizontal direction
with friction.
The authors prove that, if the final time $T$ is larger than
\begin{equation}
T_{lim} =\frac{1}{\attr} \log \Big((1+\sqrt{2}) e^{\attr^2 X} -1+\sqrt{(1+\sqrt{2}) e^{\attr^2 X} -1)^2-1}\Big),
\end{equation}
then 
the optimal control has the bang-singular-inactivated-bang structure described in equation \eqref{eq: struttura controllo}.
The corresponding trajectories satisfy PMP in normal form, with adjoint covector $\boldsymbol{p}(t)=(p_1(t),p_2(t))$, with $p_1(t)\equiv p_1^0$ for every $t\in [0,T]$.
In particular, the following relations hold

\begin{equation}  \label{tau1 es}
\wtau_1=\frac{1}{\attr} \log\Big(\frac{2}{2-p_1^0}\Big),\quad
\wtau_3-\wtau_2=\frac{1}{\attr}\log\big(1+\sqrt{2}\big),\quad 
 T-\wtau_3=\frac{1}{\attr}\log\Big(\frac{p_1^0+2(1+\sqrt{2})}{2(1+\sqrt{2})}\Big),
\end{equation}
and
\begin{equation} \label{uS es}
\wu_S(t)\equiv\sqrt{\attr p_2(0)}=\frac{p_1^0}{2}=\attr p_2(\wtau_1) \qquad \forall t\in[\wtau_1,\wtau_2].
\end{equation}
Coupling equations~\eqref{tau1 es} and \eqref{uS es}, we deduce that $p_1^0$ must be positive and smaller than 2. 
Finally, the authors prove that  along these extremals $x_2(t)$ is positive for every $t\in(0,T)$ and
\begin{equation}
x_2(t)\equiv p_2(t)=\frac{p_1^0}{2\attr} \qquad \forall t\in[\wtau_1,\wtau_2] .
\end{equation}

Using the notations of our paper, the drift $f_0$, the controlled vector field $f_1$ and the cost $\psi$ at a point $\wb x = (x_1, x_2)\in \R^2$ are respectively given by
\begin{equation} \label{eq: pullback esempio}
f_0(\wb x)=\begin{pmatrix} 
     x_2\\-\attr x_2
    \end{pmatrix},  \qquad
f_1(\wb x)=\begin{pmatrix} 
     0\\1
    \end{pmatrix}, \qquad
    \psi(\wb x)
    =x_2.
\end{equation}
Thus, the associated Hamiltonian functions have the following expressions
\begin{gather*}
F_0(\wb p,\wb x)=p_1x_2-\attr p_2 x_2, \qquad
F_1(\wb p,\wb x)=p_2, \qquad
\Phi^{\pm}(\wb p,\wb x) =  p_2  \pm|x_2|. 
\end{gather*}
Here below, we prove that Assumptions~\ref{ass:  nonzero at switch}--\ref{ass:  coerc} are met.

\subsection{Regularity assumptions}

 Since $x_2(t)$ never vanishes in $(0,T)$, then Assumptions~\ref{ass:  nonzero at switch}-\ref{ass:  non vanish} and \ref{ass:  segni} are trivially satisfied.
To verify that also Assumption~\ref{ass:  struttura arco} holds true, we are left to prove that the bang and inactivated arcs are regular.
In order to complete this task, we compute the iterated Lie brackets of the vector fields $f_0$ and $f_1$: 
\begin{equation} \label{eq: campi esempio}
f_{01}(\wb x)\equiv \begin{pmatrix} 
     -1\\ \attr     \end{pmatrix},
\qquad
f_{101} (\wb x)\equiv \begin{pmatrix} 
     0\\0
    \end{pmatrix}\qquad
\mbox{ad}^k_{f_0}f_{01} (\wb x)=\attr^k f_{01} (\wb x)\equiv \begin{pmatrix} 
     -\attr^k \\ \attr^{k+1}     \end{pmatrix}.
\end{equation}

\noindent 
\paragraph{Regularity of the first bang arc.} 
We have to prove that $\Phi^-(\boldsymbol{p}(t),\boldsymbol{x}(t))>0$ for $t\in[0,\wtau_1)$.
The claim follows directly from the computations
\begin{equation}
\Phi^-(\boldsymbol{p}(t),\boldsymbol{x}(t))=\frac{e^{-\attr t}}{4 \attr} \Big( e^{\attr t} (2-p_1^0)-2 \Big)^2,
\end{equation}
and equation~\eqref{tau1 es}.

\paragraph{Regularity along the inactivated arc.} 
We must verify that  $p_2(t)- x_2(t)<0<p_2(t)+ x_2(t)$ for $t\in(\wtau_2,\wtau_3)$.
By computations
\begin{equation}
p_2(t)\pm x_2(t) = \frac{p_1^0}{2\attr} \big( - e^{\attr(t-\wtau_2)} +2 \pm e^{-\attr(t-\wtau_2)} \big). 
\end{equation}
The claim follows from equation~\eqref{tau1 es}.

\paragraph{Regularity along the last bang arc.} 
A straightforward computation gives  $p_2(t)+x_2(t)\leq 0$ for $t\in[\wtau_3,T]$, where the equality holds if and only if $t=\wtau_3$.

\paragraph{Regularity at the switching points (Assumption~\ref{ass:  reg switch}).} 

\noindent 
At the first switching time $t=\wtau_1$, Assumption~\ref{ass:  reg switch} reads
\begin{equation}
0<
\big(F_{001}+F_{101} +L_{f_{01}}\psi-L_{f_0+f_1}L_{f_0}\psi\big) (\well_1)
=-\attr p_1^0 +\attr^2 p_2(\wtau_1) +\attr +\attr(1-\attr x_2(\wtau_1))
= \attr (2-p_1^0),
\end{equation}
which is verified thanks to equation~\eqref{tau1 es}.

\smallskip
\noindent 
At $t=\wtau_2$, the regularity condition is
\begin{equation}
0> 
\big(F_{001}-L^2_{f_0}\psi\big)(\well_2)
=-\attr p_1^0 +\attr^2 p_2(\wtau_2) -\attr^2 x_2(\wtau_2)
= -\attr p_1^0,
\end{equation}
which is trivially satisfied.

\smallskip
\noindent 
At $t=\wtau_3$, the regularity condition reads
\begin{equation}
0> 
\big(F_{01}-L_{f_0}\psi\big)(\well_3)
= -p_1^0 +\attr p_2(\wtau_3) +\attr x_2(\wtau_3)
= -p_1^0,
\end{equation}
which is verified.

\medskip
\noindent 
{\bf Assumption~\ref{ass:  sglc} (SGLC).} This Assumption is trivially satisfied, since 
\begin{align}
\mathbb{L}(\wb p, \wb x)&=F_{101}(\wb p, \wb x)+L_{f_{01}} \psi(\wb x)-L_{f_1}L_{f_0}\psi(\wb x)\\
 &= \attr-L_{f_1}(-\attr x_2)=2\attr >0 
  \qquad \forall (\wb p, \wb x)\in T^* \R^2.
\end{align}

\subsection{Second variation}

First of all, we  compute the pull-back vector fields,  by means of  
formula \eqref{eq: pullback dei campi}
 and of 
equations~\eqref{eq: pullback esempio}-\eqref{eq: campi esempio}. We obtain the following expressions:
\begin{gather}
g_t^1\equiv f_1+\frac{e^{\attr(t-\wtau_1)}-1}{\attr}f_{01}, \qquad
\dot{g}_t^1\equiv e^{\attr(t-\wtau_1)}f_{01}, \qquad t\in[\wtau_1,\wtau_2],\\
k_3=f_0-\wu_S  \frac{e^{\attr(\wtau_2-\wtau_1)}-1}{\attr}f_{01},\qquad
k=-f_1-\frac{e^{\attr(\wtau_3-\wtau_1)}-1}{\attr}f_{01}
=\begin{pmatrix}
\frac{e^{\attr(\wtau_3-\wtau_1)}-1}{\attr}\\ -e^{\attr(\wtau_3-\wtau_1)}
\end{pmatrix}.
\end{gather}

\paragraph{Admissible variations}
The pullback system \eqref{eq:zetaridotto} assumes the form
\begin{equation} \label{eq: W esempio}
\begin{cases}
\zeta_1(t)=-\int_{\wtau_1}^{t} w(s) e^{\attr(s-\wtau_1)}\ud s,\\
\zeta_2(t)=\ep_0+\attr\int_{\wtau_1}^{t} w(s) e^{\attr(s-\wtau_1)}\ud s,
\end{cases} 
\end{equation}
so that the space of extended admissible variation $\cW$ can be identified with  the following subspace of $ \R\times L^2([\wtau_1,\wtau_2],\R)$: \begin{equation}
\mathcal{W}=\Big\{(
\ep,w) \colon  \int_{\wtau_1}^{\wtau_2}w(s)e^{\attr(s-\wtau_1)}\ud s=\ep\frac{e^{\attr(\wtau_3-\wtau_1)}-1}{\attr}\Big\}.
\end{equation}

It is easy to see that $\wh{S}_t^0$ is linear with respect to the state; since both $\dot{g}_t^1$ and $k$ are constant
with respect to the basepoint, we choose  $\wh{\gamma}_T$ as a linear function of the state, so that all second derivatives of the term 
$\wh{\gamma}_T+\wh{S}_T^0-\wh{S}_{\wtau_2}^0$ are zero. 
In particular, with this choice we obtain 
$\uD^2 \big(\alpha+\wh{\gamma}_T+\wh{S}_T^0 \big)[\ep f_1]^2=\ep^2$. 

Keeping all these facts into account,
the second variation reads
\begin{equation} \label{eq: secvar esempio}
J_e\se[(
\ep,w)]^2 = \ep^2(1+\sqrt{2} \wu_S) +\int_{\wtau_1}^{\wtau_2} \attr w^2(t)-\attr w(t) \zeta_2(t) e^{-\attr(t-\wtau_1)}\ud t. 
\end{equation}

\subsubsection{Coerciveness of the second variation}

We recall that, given any subspace $\mathcal{V}\subset \cW$, $J_e\se$ is coercive on $\mathcal{W}$ if and only if it is coercive both $\mathcal{V}$ and $\mathcal{V}^{\bot}$, $\mathcal{V}^{\bot}$ denoting the orthogonal complement of $\mathcal{V}$  in $\mathcal{W}$ with respect to the bilinear form associated with $J_e\se$.
To prove the coerciveness of \eqref{eq: secvar esempio}, we  choose
\begin{equation}
\mathcal{V}=\Big\{(
0,w) \colon  \int_{\wtau_1}^{\wtau_2}w(s)e^{\attr(s-\wtau_1)}\ud s=0\Big\}.
\end{equation}
\paragraph{Coerciveness on $\mathcal{V}$.} In order to prove the claim, we apply the characterization of coerciveness given in \cite[Lemma~5]{ASZ98a}, that is, we study the LQ optimal control problem 
\begin{equation}
\min \int_{\wtau_1}^{\wtau_2} \attr w^2(t)-\attr w(t) z_2(t) e^{-\attr(t-\wtau_1)}\ud t
\end{equation}
subject to 
\begin{equation}
\begin{cases}
\dot{\wb z}(t)=w(t) e^{\attr(t-\wtau_1)} f_{01},\\
\wb z(\wtau_1)=\wb z(\wtau_2)=0,                                        
\end{cases} \qquad  w\in L^2([\wtau_1,\wtau_2])  \colon  \int_{\wtau_1}^{\wtau_2}w(t)e^{\attr(t-\wtau_1)}\ud t=0.
\end{equation}
The maximised Hamiltonian $\Hex$ associated with this problem is given by
\begin{equation} \label{eq: max Ham ex}
\Hex (\wb p,\wb z , t)= \frac{1}{4\attr}\big((\attr p_2-p_1)e^{\attr(t-\wtau_1)}+\attr z_2e^{-\attr(t-\wtau_1)} \big)^2.
\end{equation}
The solutions of the Hamiltonian system associated with \eqref{eq: max Ham ex} are given by
\begin{equation}
\begin{cases}
z_1(t)=\attr (a+b(t-\wtau_1))e^{\attr(t-\wtau_1)}, \\
z_2(t)=(a+b(t-\wtau_1))e^{\attr(t-\wtau_1)}, \\
p_1(t)\equiv p_1(\wtau_1), \\
p_2(t)=\frac{p_1(\wtau_1)}{\attr} + \big(\frac{2b+a}{\attr} +b (t-\wtau_1)\big) e^{-\attr(t-\wtau_1)},
\end{cases}
\end{equation}
for some real constants $a,b$.
The boundary conditions $\wb z(\wtau_1) = \wb z(\wtau_2) = 0$ are satisfied only for $a=b=0$, that is $\wb z(t)\equiv 0$. Thanks to \cite[Lemma~5]{ASZ98a}, this implies that $J\se_e$ is coercive
on $\mathcal{V}$.

\paragraph{Coerciveness on $\mathcal{V}^{\bot}$.}
First of all, we characterize $\mathcal{V}^{\bot}$.
The bilinear form associated with $J\se_e$  is given by the formula 
\begin{equation} \label{eq: bilinear esempio}
J^b[\delta e,\widetilde{\delta e}] = \ep \widetilde{\ep}(1+\sqrt{2}\wu_S)+\frac{1}{2}\int_{\wtau_1}^{\wtau_2}
2\attr w(t)\widetilde{w}(t)-\attr\big(w(t)\widetilde{\zeta}_2(t)+ \widetilde{w}(t)\zeta_2(t) \big)e^{-\attr(t-\wtau_1)}\ud t,
\end{equation}
with $\delta e,\wt{\delta e}\in \cW$.
Introducing $p_2(\cdot)$ as the solution of the Cauchy problem
\begin{equation}
\begin{cases}
 \dot{p}_2(t)= w(t)e^{\attr(t-\wtau_1)}, \\
 p_2(\wtau_1)=0,
\end{cases} 
\end{equation}
equation~\eqref{eq: bilinear esempio} becomes
\begin{equation}
J^b[\delta e,\widetilde{\delta e}] = \ep \widetilde{\ep}(1+\sqrt{2}\wu_S)-\attr \widetilde{\zeta}_2(\wtau_2)p_2(\wtau_2)
+\frac{1}{2}\int_{\wtau_1}^{\wtau_2}
\attr \widetilde{w}(t)\Big( 2w(t) -\zeta_2(t)e^{-\attr(t-\wtau_1)}+p_2(t) e^{-\attr(t-\wtau_1)}\Big)\ud t. 
\end{equation}
It is immediate to see that an admissible variation $\delta e$ belongs to $\mathcal{V}^{\bot}$ if and only if
$e^{-\attr(t-\wtau_1)} \Big( 2w(t) -\zeta_2(t)e^{-\attr(t-\wtau_1)}+p_2(t) e^{-\attr(t-\wtau_1)}\Big) $
does not depend on $t$ (indeed, the orthogonal complement to zero-mean functions in $L^2([\wtau_1,\wtau_2])$
is the space of constant functions); we thus set
\begin{equation} \label{eq: caratterizzazione ortogonale esempio}
C_{\ep}=e^{-\attr(t-\wtau_1)} \Big( 2w(t) -\zeta_2(t)e^{-\attr(t-\wtau_1)}+p_2(t) e^{-\attr(t-\wtau_1)}\Big) ,
\end{equation} so that 
\begin{equation} \label{eq: Jbot esempio}
J\se[\delta e]^2|_{\mathcal{V}^{\bot}} = \ep^2(1+\sqrt{2}\wu_S)-\ep\attr p_2(\wtau_2) e^{(\wtau_3-\wtau_1)}
+\frac{C_{\ep}}{2}\ep (e^{(\wtau_3-\wtau_1)}-1). 
\end{equation}

From equation~\eqref{eq: caratterizzazione ortogonale esempio}, combined with \eqref{eq: W esempio}, we can deduce
that $w(\cdot)$ is smooth. 
Differentiating \eqref{eq: caratterizzazione ortogonale esempio} with respect to $t$, we obtain
\begin{equation}\label{eq: w linear}
 \dot{w}(t) e^{-\attr(t-\wtau_1)}-\attr w(t)e^{-\attr(t-\wtau_1)} +\attr  \zeta_2(t) e^{-2\attr(t-\wtau_1)} =0.
\end{equation}
Multiplying \eqref{eq: w linear} by $e^{2\attr(t-\wtau_1)}$ and differentiating again, we obtain 
$\ddot{w}(t)=0$,
i.e.~$w(\cdot)$ is an affine function. 
Plugging into the previous equations we obtain
\begin{equation}
C_{\ep}=\ep\frac{2\sqrt{2}+\attr (\wtau_2-\wtau_1)}{\attr (\wtau_2-\wtau_1)}
\end{equation}
and
\begin{equation}
p_2(\wtau_2)=
\frac{\ep}{\attr^2(\wtau_2-\wtau_1)}
 \Big(  -e^{-\attr(\wtau_2-\wtau_1)} \left(2\sqrt{2}+(1+\sqrt{2}) \attr (\wtau_2-\wtau_1)\right)+
 2\sqrt{2}+\attr(\wtau_2-\wtau_1)\Big).
\end{equation}
Substituting these expressions into \eqref{eq: Jbot esempio}, we obtain
\begin{equation}
J\se[\delta e]^2|_{\mathcal{V}^{\bot}}=\ep^2 \left( 2+\sqrt{2} +\frac{2}{\attr(\wtau_2-\wtau_1)}+\sqrt{2}\wu_S\right),
\end{equation}
which is  positive whenever $\ep\neq 0$.

\subsubsection{The over-maximised Hamiltonian}

Although the expressions of the Hamiltonians $H_0$ and $K$ 
are not necessary 
to deduce the optimality of the reference extremal, we provide their explicit formulas,  to give an insight on their construction.

Since, along the extremal, $x_2 >0$, we do not use the absolute value in the expressions of $\Phi^-$. 
First of all, we see that $\Sigma^-=\{(\wb p,\wb x) \colon p_2-x_2=0\}$ and $S^-=\{(\wb p,\wb x) \colon  p_2=x_2 \mbox{ and } p_1=\attr(p_2+x_2)\}$.
Straightforward computations yield 
\begin{equation}
\vartheta(\wb p, \wb x)=\frac{p_1}{2\attr}-\frac{x_2+p_2}{2}, \qquad 
H_0(\wb p,\wb x)=\frac{1}{4\attr} \big(p_1-\attr(p_2-x_2)\big)^2.
\end{equation}
In particular, we obtain that
\begin{equation}
H_0(\wb p,\wb x) -F_0(\wb p,\wb x)= \frac{1}{4\attr} \Big( (p_1-2\attr x_2)^2 +4\attr^2 x_2(p_2-x_2) \Big).
\end{equation}
This formula shows that $H_0\geq F_0$ only for $(\wb p,\wb x)\in \Sigma^-$, with equality for  $(\wb p,\wb x)\in S^-$.
Finally $\nu(\wb p, \wb x) = \dfrac{p_1}{2}$ for any $(\wb p, \wb x) \in T^*\R^2$, so that 
\begin{equation}
K(\wb p, \wb x)  = \frac{1}{4\attr} \big(p_1-\attr(p_2-x_2)\big)^2 + \dfrac{p_1}{2}(p_2 - x_2).
\end{equation}

\section{Conclusions}

In this paper, we develop the analysis of sufficient optimality conditions for 
generalised $L^1$-optimal control problems started in \cite{CP19}. 
In particular, we consider the case in which the reference extremal contains a singular arc. As already observed in the Introduction, this fact  brings
significant technical difficulties, in particular the necessity of computing the second variation of a singular arc for a Bolza problem, which, to our 
knowledge, has not been done before.

We believe that \cite{CP19} and the current paper, altogether, provide a solid basis 
for the study of sufficient optimality conditions for problems
of the form \eqref{eq: cost}-\eqref{eq: contr sys}: indeed, they give an insight of how to figure out more complex cases (as the concatenation of
several bang, singular and inactivated extremals).

Two only issues are left over:  the possibility of bang-bang concatenations (here prevented by Assumption~\ref{ass: nonzero at switch}), that, in our
opinion, can be treated providing minor changes to the existing results, taking advantage of the techniques developed, for instance, in~\cite{PSp15}; the case of degenerate singular arcs, which, even if it is 
a non-generic case, is nevertheless theoretically challenging. 
The authors are planning to  consider this last situation.

\appendix

\section{Sketch of the computation of the second variation}
\label{app secvar}

We first recall that  the reference extremal $\wla$, associated with the control \eqref{eq: struttura controllo},  satisfies the following equation
\begin{equation} \label{eq: solution lambda}
\wla(t) =\Big(\well_1+\int_{\wtau_1}^t  |u(s)|\  d|\wpsi{s}|(\wxu) \ud s \Big)  \circ \wS{t*}^{-1}
=\Big(\well_1 + \ud \wS{t}^0{(\wxu)} \Big)  \circ \wS{t*}^{-1}
\qquad \forall t\in[0,T]. 
\end{equation}

Moreover  the cost realised by the trajectory $\boldsymbol{\xi}$ associated with the control $u$  can be written as 
$J(u)=\widehat{S}_T^{\circ}(\boldsymbol{\eta}_T)-\widehat{S}_0^{\circ}(\boldsymbol{\eta}_0)+\widehat{\gamma}_T(\eta_T)
+\widehat{\gamma}_0(\eta_0)$, thanks to equation \eqref{eq: costo J}.
Further, we notice that the variations $\delta u=u-\widehat{u}$ can be encoded in three terms, that is, the variation $\delta v$ of the control along the singular arc, 
and the variations $\ep_3=(\tau_3-\wtau_2)-(\wtau_3-\wtau_2)$ and $\varepsilon_4=(T-\tau_3)-(T-\wtau_3)$ of the length of the third and fourth arc, related by the constraint $\ep_3 + \ep_4 = 0$.
In particular, there is no variation of the terms $-\widehat{S}_0^{\circ}(\boldsymbol{\eta}_0)+\widehat{\gamma}_0(\eta_0)$. The second variation $J''[\delta u]$ is thus given by  
\begin{equation}
\frac{\partial^2 J}{\partial u^2} = \frac{\partial^2 \eta^0_T}{\partial u^2} +
\left( \uD^2 \widehat{S}^0_T(\wq_0) + \uD^2 \wh{\gamma}_T (\wq_0) \right) \left[\frac{\partial \eta_T}{\partial u}\right]^2
+\scal{ \ud \widehat{S}^0_T(\wq_0) + \ud \wh{\gamma}_T (\wq_0)}{ \frac{\partial^2 \eta_T}{\partial u^2}}
\end{equation}
where, for the differential of $\wh S^0_t$,  we have used the notation  established in Remark~\ref{re:notazione} and where the derivatives with respect to $u$ have to be intended as  the derivatives with respect to the variables $(v,\ep_3,\ep_4)$, which can be computed using the equations 
\begin{align}
\eta_T^0 &=   \int_{\wtau_1}^{\wtau_2} \!\! \big( 
	(v(s)-\wu(s)) \wh{\psi}_s(\eta_s)-L_{\dot{\eta}_s} \wh{S}_s^0(\eta_s)
	\big)  \ud s %
	 -\int_{\wtau_2}^{\wtau_3} \!\! 
L_{\dot{\eta}_s} \wh{S}_s^0(\eta_s) \ud s  %
+ 
\int_{\wtau_3}^{T} \!\! \Big( \frac{\ep_4 
}{T-\wtau_3} \wh{\psi}_s(\eta_s)-L_{\dot{\eta}_s} \wh{S}_s^0(\eta_s) \Big) \ud s, \\
\eta_T &= \wq_1+\int_{\wtau_1}^{\wtau_2} (v(s)-\wu(s)) g_s^1(\eta_s)\ud s + \int_{\wtau_2}^{\wtau_3} \frac{\ep_3 
}{\wtau_3-\wtau_2} k_3(\eta_s) \ud s +
\int_{\wtau_3}^{T} \frac{\ep_4
}{T-\wtau_3} k_4(\eta_s) \ud s.
\end{align}
We first show  that  the first order approximation of $J$, evaluated at $(v, \ep_3, \ep_4) =  (\wh u, 0,0)$  is null, so that the second variation is intrinsically well defined. We point out that, thanks to  \eqref{eq: solution lambda}, $\ud \wh{S}_T^0(\wq_1) + \ud\wh\gamma_T(\wq_1) = - \well_1$, so that 
\begin{equation}
\begin{split}  
\Big\langle\dfrac{\partial J}{\partial v},\delta v(\cdot)\Big\rangle &=  
\dfrac{\partial \eta_T^0}{\partial v}\delta v(\cdot)  - \scal{\well_1 }{\dfrac{\partial \eta_T}{\partial v} }  \delta v(\cdot) 
=    \int_{\wtau_1}^{\wtau_2} \delta v(s) \left( 
\wh{\psi}_s(\wq_1) - L_{g^1_s} \wh{S}_s^0(\wq_1)
\right) \ud s + 
\int_{\wtau_1}^{\wtau_2} \delta v(s) 
\scal{\well_1}{g^1_s(\wq_1)}
\ud s \\
&=    \int_{\wtau_1}^{\wtau_2} \delta v(s) \left( 
\psi(\wxi(s)) - L_{g^1_s} \wh{S}_s^0(\wq_1)
- F_1(\wla(s)) + L_{g^1_s} \wh{S}_s^0(\wq_1)
\right)\ud s  = 0 .
\end{split}
\end{equation}
Since only variations where $\ep_4 = -\ep_3 $ are admissible, it suffices to show that  $\dfrac{\partial J}{\partial \ep_3}  = \dfrac{\partial J}{\partial \ep_4} $.
Notice that 
\begin{equation}\label{eq:truccoquartoarco}
L_{k_4} \big(\wh{S}^0_T-\wh{S}^0_{\wtau_2}\big)(\wq_1)=  L_{k_4} \big(\wh{S}^0_T-\wh{S}^0_{\wtau_3}\big)(\wq_1)=\int_{\wtau_3}^TL_{k_4} \wh{\psi}_s(\wq_1) \ud s= \wh{\psi}_T(\wq_1)-\wh{\psi}_3(\wq_1),
\end{equation}
since $\wh{S}^0_{\wtau_2}=\wh{S}^0_{\wtau_3}$ and, for every $s\in[\wtau_3,T]$, $\frac{\ud}{\ud s}\wh{\psi}_s=L_{k_4}\wh{\psi}_s$.
Finally, we get
\begin{align}
\begin{split}  
\dfrac{\partial J}{\partial  \ep_3} &= 
 \dfrac{\partial \eta_T^0}{\partial  \ep_3}  + \scal{\well_1 }{\dfrac{\partial \eta_T}{\partial  \ep_3} }  
 =   - L_{k_3} \wh{S}_{\wtau_2}^0(\wq_1) - \scal{\well_1}{k_3(\wq_1)} 
 =   - L_{k_3} \wh{S}_{\wtau_3}^0(\wq_1) - \scal{\well_1}{k_3(\wq_1)} 
\end{split}\\
\begin{split}  
\dfrac{\partial J}{\partial \ep_4} &= 
\dfrac{\partial \eta_T^0}{\partial\ep_4}  + \scal{\well_1 }{\dfrac{\partial \eta_T}{\partial \ep_4} }  
=   - L_{k_4} \wh{S}_{T}^0(\wq_1) + \wpsi{T}(\wq_1) - \scal{\well_1}{k_4(\wq_1)}
=   - L_{k_4} \wh{S}_{\wtau_3}^0(\wq_1) + \wpsi{\wtau_3}(\wq_1) - \scal{\well_1}{k_4(\wq_1)}
\end{split}
\end{align}
so that 
\begin{equation}
\begin{split}
 \dfrac{\partial J}{\partial\ep_3} -  \dfrac{\partial J}{\partial \ep_4} 
&=   L_{k_4 - k_3} \wh{S}_{\wtau_3}^0(\wq_1) 
 - \wpsi{\wtau_3}(\wq_1) - \scal{\well_1}{(k_4 - k_3)(\wq_1)} \\ 
 &=   L_{k_4 - k_3} \wh{S}_{\wtau_3}^0(\wq_1) 
 - \wpsi{\wtau_3}(\wq_1) 
 - \scal{- \uD \wh{S}^0_{\wtau_3} \wh{S}_{\wtau_3 *}^{-1} +  \well_3}{f_1(\wq_3)} 
 = - \Phi^+(\well_3) = 0.
\end{split}
\end{equation}

Let us now compute the second variation.
By some long computations, it is possible to obtain
\begin{equation}\label{eq:varsecappendix}
\begin{split}
J\se[\delta v,\varepsilon_3,\varepsilon_4]^2 &:= 
 \int_{\wtau_1}^{\wtau_2}  \int_{\wtau_1}^{s} \delta v(s) \delta v(r) 
L_{g^1_r}\left( 
{ \wpsi{s}}  + L_{g^1_s}
{ \left(\wh\gamma_T + \wS{T}^0 - \wS{s}^0\right)} \right)(\wxu) \ud r\ud s \\
& + \ep_3 \int_{\wtau_1}^{\wtau_2}  \delta v(s) \liederder{g^1_s}{k_3}{\left(\wh\gamma_T +  \wS{T}^0 - \wS{\wtau_2}^0\right)}\ud s 
 + \ep_4 \int_{\wtau_1}^{\wtau_2} \delta v(s) \lieder{g^1_s}{\left( \abs{\wpsi{T}} + L_{k_4}{\wh\gamma_T} \right) } \ud s \\
 & + \dfrac{\ep_3^2}{2}\liederdo{k_3}{ \left(\wh\gamma_T + \wS{T}^0 - \wS{\wtau_2}^0\right)}  
+ \dfrac{\ep_4^2}{2}L_{k_4}\left( \abs{\wpsi{T}} +  L_{k_4}{ \wh\gamma_T }  \right)(\wxu) \\
& + {\ep_3\ep_4} L_{k_3}{
\left(	\abs{\wpsi{\wtau_3}} + L_{k_4}{ \wh\gamma_T   } \right)
}(\wxu) \\
&= 
 \int_{\wtau_1}^{\wtau_2}  \delta v(s)
L_{\delta \eta_r}\left( 
{ \wpsi{s}}  + L_{g^1_s}
{ \left(\wh\gamma_T + \wS{T}^0 - \wS{s}^0\right)} \right)(\wxu) \ud s \\
& - \ep  \liederder{\delta \eta(\wtau_2)-\delta \eta(\wtau_1)}{k_3}{\left(\wh\gamma_T +  \wS{T}^0 - \wS{\wtau_2}^0\right)} 
 + \ep  \lieder{\delta \eta(\wtau_2)-\delta \eta(\wtau_1)}{\left( \abs{\wpsi{T}} + L_{k_4}{\wh\gamma_T} \right) } \ud s \\
 & + \dfrac{\ep^2}{2}\liederdo{k_3}{ \left(\wh\gamma_T + \wS{T}^0 - \wS{\wtau_2}^0\right)}  
+ \dfrac{\ep^2}{2}L_{k_4}\left( \abs{\wpsi{T}} +  L_{k_4}{ \wh\gamma_T }  \right)(\wxu) \\
& - {\ep^2} L_{k_3}{
\left(	\abs{\wpsi{\wtau_3}} + L_{k_4}{ \wh\gamma_T   } \right)
}(\wxu),\\
\end{split}
\end{equation}
where 
we set $\ep_4=-\ep_3=\ep$.  
Integrating backward in time system \eqref{eq: delta eta}, and applying the constraints, we see that $\delta \eta_{\wtau_2} -\delta \eta_{\wtau_1}=-\ep k(\wq_1)$.
Finally, plugging equation \eqref{eq:truccoquartoarco} into \eqref{eq:varsecappendix} we obtain equation~\eqref{eq: secvar 1}.

\section{Some technical results and proofs} \label{app:  proofs}

\begin{proposition}\label{prop: preserva Phi}
For every smooth function $\upsilon (t,\ell) \colon  \R\times T^*M \to \R$,  the Hamiltonian flow 
associated with $H_0(\ell)+\upsilon(t,\ell) \Phi^-(\ell)$ preserves $\vPhi-$ on $\Sigma^-$, that is,   
if $\eulF_t$ is the flow associated with $H_0(\ell)+\upsilon(t,\ell) \Phi^-(\ell)$ from  $\wtau_1$ to time $t$, then
\begin{equation} \label{eq: preserva Phi}
\eulF_{t*}^{-1} \vPhi-\circ \eulF_t(\ell)=\vPhi-(\ell),
\qquad \forall \ell\in \cO_1.
\end{equation}

\end{proposition}

\begin{proof}
We notice that, for every $\ell\in \Sigma^-$,
\begin{align*}
\frac{\partial }{\partial t} \eulF_{t*}^{-1} \vPhi-\circ \eulF_t(\ell) &=
\eulF_{t*}^{-1} [\vec{H}_0+\upsilon(t,\cdot)\vPhi-+\Phi^-(\cdot) \vec{\upsilon}(t,\cdot), \vPhi-]\circ \eulF_t(\ell)\\
&=\eulF_{t*}^{-1} [\vec{H}_0, \vPhi-]\circ \eulF_t(\ell),
\end{align*}
since $\Phi^-(\eulF_t(\ell))=0$.

Using equation~\eqref{eq: wF}, for every $\ell\in \Sigma^-$ we have
\begin{align}
[\vec{H}_0,\vPhi-](\ell)&= [\exp(-t\vPhi-)_*\vec{F}_0\circ \exp(t(\ell)\vPhi-)|_{t=\vartheta(\ell)},\vPhi-]|_{\ell}\\
&=[\exp(-t\vPhi-)_*\vec{F}_0\circ \exp(t\vPhi-),\vPhi-](\ell)|_{
t=\vartheta(\ell) 
}\\
&+\big(\langle d\vartheta,\vPhi{-}\rangle[\vec{F}_0,\vPhi-]\big) (\exp(\vartheta(\ell)\vPhi-))\\
&=\big(1+\langle d\vartheta,\vPhi{-}\rangle)[\vec{F}_0,\vPhi-]\big)|(\exp(\vartheta(\ell)\vPhi-)),
\end{align}
which is null by \eqref{dtheta}.

\end{proof}

\subsubsection*{Proof of Lemma~\ref{lemma:  K e H uguale}}

Set $\cG_t=\wh{\cH}_t^{-1}\circ\cK_t$ and notice that $\cG_t(\well_1)=\well_1$ for every $t\in[\wtau_1,\wtau_2]$.
For every $\ell\in\Sigma^-$, 
\begin{equation}
\frac{\partial}{\partial t}\cG_t(\ell) =\big((\nu-\wu(t))\wh{\cH}_{t*}\vPhi-\big) \circ \wh{\cH}_t|_{\cG_t(\ell)}=\big(\nu(\wh{\cH}_t\circ\cG_t(\ell))-\wu(t)\big)\vPhi-(\cG_t(\ell)),
\end{equation}
thanks to Proposition~\ref{prop: preserva Phi}.
Since $\vPhi-$ is tangent to $\Lambda_1$,  we obtain that 
$\cG_{t}(\Lambda_1)\subset\Lambda_1$. $\cG_{t*}(T_{\well_1}\Lambda_1) = T_{\well_1}\Lambda_1$ for every $t\in[\wtau_1,\wtau_2]$,
that is,  $\wh{\cH}_{t*}(T_{\well_1}\Lambda_1)=\cK_{t*}(T_{\well_1}\Lambda_1)$  and claim (1) is proved.

By simple computations, we can prove that $\cG_t$ is the Hamiltonian flow associated with the Hamiltonian 
$G_t=(K-\wh H_t)\circ \wh{\cH}_t$. In particular, from the fact that $dG_t|_{\well_1}=0$, we obtain that 
\begin{equation} \label{G secondo}
G_t\se :=\frac{1}{2}\uD^2 G_t|_{\well_1} =\frac{1}{2} \big(d\Phi^-\otimes d\nu+d\nu \otimes d\Phi^-  \big)|_{\wla(t)}
[\wh{\cH}_{t*}\cdot]^2 =\langle d\Phi^-|_{\wla(t)},\wh{\cH}_{t*}\cdot\rangle\langle d\nu|_{\wla(t)},\wh{\cH}_{t*}\cdot\rangle
\end{equation}
is a well  defined Hamiltonian function on $T_{\well_1}(T^*M)$,
and that $\cG_{t*}$ is the Hamiltonian flow associated with $G_t\se$, see \cite{mrr94}.

We now restrict ourselves to vectors $\dl\in T_{\well_1} \Sigma^-$. 
By definition, and using the fact that $\langle d\Phi^-|_{\wla(t)},\wh{\cH}_{t*}\dl\rangle=0$,
it follows that  $\vec{G}_t\se(\dl)=\langle d\nu|_{\wla(t)},\wh{\cH}_{t*}\dl\rangle
\overrightarrow{\langle d\Phi^-|_{\wla(t)},\wh{\cH}_{t*}\dl\rangle}$.
To compute this quantity, we set $\varphi(\dl)=\langle d\Phi^-|_{\wla(t)},\wh{\cH}_{t*}\dl\rangle$, and choose some vector 
$X\in T_{\dl} (T_{\well_1}(T^*M))\simeq T_{\well_1}(T^*M)$; then
\begin{align*}
\boldsymbol{\sigma}_{\dl}(X,\vec{\varphi}(\dl)) = \langle d\varphi|_{\dl},X\rangle= \langle d\Phi^-|_{\wla(t)},
\wh{\cH}_{t*}X \rangle&=
\boldsymbol{\sigma}_{\wla(t)} \big(\wh{\cH}_{t*}X,\vPhi-(\wh{\cH}_{t}(\well_1))\big)\\
&=\boldsymbol{\sigma}_{\well_1} \big(X,\wh{\cH}_{t*}^{-1}\vPhi-(\wh{\cH}_{t}(\well_1)) \big),
\end{align*}
so that,  by \eqref{eq: preserva Phi},
\begin{equation} \label{eq: vec G''}
\vec{G}_t\se(\dl)= \scal{ d\nu|_{\wla(t)}}{\wh{\cH}_{t*}\dl} \vPhi-(\well_1).  
\end{equation}
Let us  now assume that $\ker (\pi\wh\cH_t)_{*}|_{T_{\well_1}\Lambda_1}=0 $ for  some  $t\in[\wtau_1,\wtau_2]$.

Set $\lambda(t) :=\cG_{t*}\dl$, for some $\dl\in T_{\well_1}\Lambda_1$. Thanks to \eqref{eq: tangent split}, 
there exist a unique $\dl_S\in T_{\well_1}S^-$ and a unique $a\in \mathbb{R}$ 
such that $\dl=\dl_S+a\vPhi-(\well_1)$. From equation \eqref{eq: vec G''}, we obtain that $\lambda(t)=\dl_S+
\mu(t) \vPhi-(\ell_1)$, for some real function $\mu(\cdot)$ satisfying $\mu(\wtau_1)=a$.

In particular, this implies that 
\begin{equation}
\cK_{t*} \dl = \cK_{t*} \dl_S +a \vPhi- (\wla(t))  =\wh{\cH}_{t*}\big( \dl_S+\mu(t) \vPhi-(\well_1)\big).  
\end{equation}
Thus, if  $(\pi\cK_{t})_{*} \dl=0$, then $(\pi\wh{\cH}_t)_{*}\big( \dl_S+\mu(t) \vPhi-(\well_1)\big)=0$,
which implies, by hypothesis, that $\dl_S+\mu(t) \vPhi-(\well_1)=0$, that is, $\dl_S=0$ and $\mu(t)= 0$. By \eqref{G secondo} $\mu(t) = a + \int_{\wtau_1}^t \scal{ \ud\nu|_{\wla(s)}}{\wh{\cH}_{s*}\dl}\ud s \equiv a$ since, by construction, $\nu$ is constant along the integral lines of $\vPhi{-}$. Thus $a = 0$, so that claim (2) is proved. 
\hfill $\square$ \medskip

\section{Useful formulas}
In this section, we recall classical formulas of differential geometry that we extensively use throughout the paper.

Let $f,g$ be two vector fields on some manifold $M$. Then, for every $t$ for which $\exp(tf)$ is defined, it holds
\begin{equation} \label{eq: pullback dei campi} 
\frac{d}{dt} \exp(-tf)_* g \circ \exp(tf)=
\exp(-tf)_* [f,g] \circ \exp(tf).
\end{equation}

Let $F,G  \colon T^*M \to \mathbb{R}$ be some Hamiltonian functions, and denote, as usual, with the script their flow from some time $t_0$.
Then 
\begin{equation}
L_{\vec{F}} G=\langle dG,\vec{F}\rangle= \boldsymbol{\sigma}(\vec{F},\vec{G}) , \qquad  
[\vec F,\vec G]
=\overrightarrow{\boldsymbol{\sigma}(\vec F, \vec G)}  
, \qquad 
\cF_{t*}^{-1}\vec{G}|_{\cF_{t}(\ell)}=\overrightarrow{G\circ \cF_t}|_{\ell}.
\label{eq: vettori trasporti}
\end{equation}

\bibliographystyle{alpha}
\bibliography{bibliocompleta}
\end{document}